\documentclass[final,leqno]{siamltex704}
\usepackage{amsmath}
\usepackage{graphicx}
\usepackage{mathrsfs}
\usepackage{bbm}
\usepackage{bm}
\usepackage{float}
\usepackage{cases}
\usepackage{amsfonts,amssymb}
\usepackage{dsfont}
\usepackage{pifont}
\usepackage{tikz}
\usepackage{wrapfig} % Allows in-line images
\usepackage{hyperref}
\usepackage{multirow}
\usepackage{lineno}
\usepackage{mathtools}
\usepackage{subfigure}
\usepackage{color}
\usepackage{dashrule}
\usepackage{amsmath}
\makeatletter
\renewcommand*\env@matrix[1][\arraystretch]{%
  \edef\arraystretch{#1}%
  \hskip -\arraycolsep
  \let\@ifnextchar\new@ifnextchar
  \array{*\c@MaxMatrixCols c}}
\makeatother

\newcommand{\vertiii}[1]{{\left\vert\kern-0.25ex\left\vert\kern-0.25ex\left\vert #1
    \right\vert\kern-0.25ex\right\vert\kern-0.25ex\right\vert}}
\def\S{{\mathfrak s}}
\def\T{{\mathcal T}}
\def\E{{\mathcal E}}

\def\O{{\mathcal O}}

\def\bn{{\bf n}}

\def\bu{\bm{u}}
\def\bbf{\bm{f}}
\def\pT{{\partial T}}
\def\3bar{{|\!|\!|}}

\newtheorem{FD-algorithm}{5-Point Finite Difference Algorithm}[section]

\title{Simplified weak Galerkin and New Finite Difference Schemes for the Stokes Equation}

\author{Yujie Liu\thanks{School of Data and Computer Science, Sun Yat-sen University, Guangzhou, 510275, China (liuyujie5@mail.sysu.edu.cn). The research of Liu was partially supported by Guangdong Provincial Natural Science Foundation (No. 2017A030310285), Shandong Provincial natural Science Foundation (No. ZR2016AB15) and Youthful Teacher Foster Plan Of Sun Yat-Sen University (No. 171gpy118),} \and Junping Wang
\thanks{Division of Mathematical Sciences, National Science Foundation, Alexandria, VA 22314 (jwang@nsf.gov). The research of Wang was supported in part by the NSF IR/D program, while working at National Science Foundation. However, any opinion, finding, and conclusions or recommendations expressed in this material are those of the author and do not necessarily reflect the views of the
National Science Foundation,}
}

\begin{document}
%\linenumbers
\maketitle

\begin{abstract}
This article presents a simplified formulation for the weak Galerkin finite element method for the Stokes equation without using the degrees of freedom associated with the unknowns in the interior of each element as formulated in the original weak Galerkin algorithm. The simplified formulation preserves the important mass conservation property locally on each element and allows the use of general polygonal partitions. A particular application of the simplified weak Galerkin on rectangular partitions yields a new class of 5- and 7-point finite difference schemes for the Stokes equation. An explicit formula is presented for the computation of the element stiffness matrices on arbitrary polygonal elements. Error estimates of optimal order are established for the simplified weak Galerkin finite element method in the $H^1$ and $L^2$ norms. Furthermore, a superconvergence of order $O(h^{3/2})$ is established on rectangular partitions for the velocity approximation in the $H^1$ norm at cell centers, and a similar superconvergence is derived for the pressure approximation in the $L^2$ norm at cell centers. Some numerical results are reported to confirm the convergence and superconvergence theory.
\end{abstract}



\pagestyle{myheadings}

%
%\noindent {\bf Mathematics Subject Classification (2010)} 65N15;
%65N30; 41A30

\section{Introduction}
This paper is concerned with the development of a simplified formulation for the weak Galerkin finite element method for the Stokes equation developed in \cite{wy-stokes}. For simplicity, consider the following $2D$ Stokes problem which seeks a velocity field $\bm{u}$ and a pressure unknown $p$ satisfying
\begin{equation}\label{Equ.Stokes}
\left\{
\begin{array}{ccl}
-\Delta \bm{u}+ \nabla p &=\bm{f}  &{\rm in}\ \Omega,\\
{\rm div} \bm{u}&=0                &{\rm in}\ \Omega,\\
 \bm{u}&=0                         &{\rm on}\ \Gamma,
\end{array}
\right.
\end{equation}
where $\Omega$ is a bounded polygonal domain with boundary $\Gamma=\partial\Omega$, ${\bf f}=(f^{(1)},f^{(2)})^t$ represents external source, and the pressure function $p$ is assumed to have mean value zero; i.e., $\int_{\Omega} p d\Omega =0$.

The weak Galerkin finite element method (WG-FEM) is a natural extension of the standard conforming finite element method where differential operators are approximated as discrete distributions or discrete weak derivatives. WG-FEM has the flexibility of dealing with discontinuous elements while sharing the same weak formulation with the classical conforming finite element methods. Following its first development in \cite{WangYe_2013,wy3655} for second order elliptic equations, the method has gained a lot attention and popularity and has been applied to various PDEs including the Stokes equation, biharmonic equation, elasticity equation, and Maxwell's equations, see \cite{WangWang_2016,RXZZ-JCAM-2016,LiDanWW} and the references therein for more details.

In this paper, we propose a simplified formulation for the weak Galerkin finite element method for the Stokes equation. In the simplified weak Galerkin (SWG), the degree of freedoms involves only those on the element boundary; i.e., the unknowns associated with the interior of each element in the original weak Galerkin are not used at all in SWG. This simplified formulation still preserves the important mass conservation property locally on each element and allows the use of general polygonal partitions. The stability and convergence is established by following the general framework developed in \cite{wy-stokes}, but with a non-trivial modification due to the absence of the interior unknowns employed in the original weak Galerkin. A comprehensive formula for the computation of the element stiffness matrices is presented for a better understanding and dissemination of the method. As a particular case study, the global stiffness matrix for the simplified weak Galerkin is assembled on uniform rectangular partitions, yielding a new class of 5- and 7-point finite difference schemes for the Stokes equation. Furthermore, a superconvergence is established in the $H^1$ norm for the velocity approximation at cell centers, as well as for the pressure approximation in the $L^2$ norm at cell centers on (nonuniform) rectangular partitions. A section is devoted to the description of the new class of 5- and 7-point finite difference schemes in order to facilitate the practical use of this scheme. Due to the connection between the finite difference scheme and the simplified weak Galerkin finite element method, the convergence theory developed for SWG can be easily extended to the new finite difference schemes.
The local conservation property is also inherited by the finite difference scheme. The SWG method indeed provides a way to generalize the new finite difference scheme from regular Cartesian grids to irregular polygonal grids.

There are existing finite difference schemes in literature for solving the Stokes equations \cite{Welch-MAC,Strikwerda_1984,LTF-IJNMF-1995}. The MAC scheme \cite{Welch-MAC} is one of the finite difference methods that has enjoyed a lot attention and popularity in computational fluid dynamics. Nicolaides \cite{Nicolaides_1989} showed that the MAC method can be seen as a special case of the co-volume formulation. Han and Wu \cite{Han_Wu_1998} derived the MAC scheme from a newly proposed mixed finite element method, in which the two components of the velocity and the pressure are defined on different meshes (an approach that resembles the MAC idea). Recently, Kanschat \cite{Kanschat2008} showed that the MAC scheme is algebraically equivalent to the divergence-conforming discontinuous Galerkin method with a modified interior penalty formulation. In \cite{Minev_2008}, Minev demonstrated that the MAC scheme can be interpreted within the framework of the local discontinuous Galerkin (LDG) methods. The 5- and 7-point finite difference schemes of this paper are based on different stencils from the MAC scheme, and therefore represent a new class of numerical schemes which are stable and have error estimates of optimal order.

The paper is organized as follows: In Section \ref{section.SWGpoly}, we present a simplified version of the weak Galerkin finite element method on arbitrary polygonal partitions. In Section \ref{Section:stiffnessmatrix}, we present a formula with comprehensive derivation for the computation of the element stiffness matrices in SWG. In Section \ref{sectionFDSWG}, we apply the SWG scheme to uniform rectangular partitions and derive a class of 5- and 7-point finite difference schemes. In particular, Subsection \ref{subsection-5point} is devoted to a presentation of the class of 5-point finite difference scheme for the model problem \eqref{Equ.Stokes}. In Section \ref{Section:Stability}, we provide a stability analysis for the simplified weak Galerkin. In Section \ref{sectionEEStokes}, we derive some error estimates for the SWG approximations in various Sobolev norms. In Section \ref{sectionSuperC}, a superconvergence result is developed for the numerical velocity and pressure arising from rectangular partitions. Finally, in Section \ref{numerical-experiments}, we present a few numerical results to demonstrate the efficiency and accuracy of SWG.
%In particular, we will first verify the theoretical convergence through couple of testing examples, and then demonstrate the power of the new scheme in scientific computing through the Stokes' model.

\section{A Simplified Weak Galerkin}\label{section.SWGpoly}
Assume that the domain is of polygonal type and is partitioned into unstructured polygons $\T_h=\{T\}$. Let $T\in \T_h$ be a polygonal element of $N$ sides (e.g., a hexahedron shown in Fig. \ref{fig.hexahedron}). Denote by $\{e_i\}_{i=1}^{N}$ the edge set of $T$. Denote by $M_i$ the midpoint of the edge $e_i$, and $\bn_i$ the outward normal direction of $e_i$; see Fig. \ref{fig.hexahedron}.

Given any piecewise constant function $u_b$ on the boundary of $T\in \T_h$; i.e.,
$u_b|_{e_i} =u_{b,i}$, we define the weak gradient of $u_b$ on $T$ by
\begin{equation}\label{DefWGpoly}
\nabla_w u_b:=\displaystyle\frac{1}{|T|}\sum_{i=1}^N u_{b,i}|e_i|\bf{n_i},
\end{equation}
where $|e_i|$ is the length of the edge $e_i$, $|T|$ is the area of the element $T$. For convenience, we denote by $W(T)$ the space of piecewise constant functions on $\pT$. The global finite element space $W(\T_h)$ is constructed by patching together all the local elements $W(T)$ through single values on interior edges. The subspace of $W(\T_h)$ consisting of functions with vanishing boundary value is denoted as $W_0(\T_h)$.

We use the conventional notation of ${P}_j(T)$ for the space of polynomials of degree $j\ge 0$ on $T$. For each $u_b\in W(T)$, we associate it with a linear extension in $T$, denoted by $\S(u_b)\in {P}_1 (T)$, satisfying
\begin{equation}\label{Def.extension}
\sum_{i=1}^{N}(\S(u_b)(M_i) -u_{b,i})v(M_i)|e_i|=0\quad \forall\; v\in {P}_1(T).
\end{equation}
It is easy to see that $\S(u_b)$ is well defined by \eqref{Def.extension}, and its computation is simple and local. In fact, $\S(u_b)$ can be viewed as an extension of $u_b$ from $\partial T$ to $T$ through a least-squares fitting.

\begin{figure}[!h]
\begin{center}
\begin{tikzpicture}[rotate=0, scale =2.0]
    %define the points of triangle
    %you can also use \coordinate to define these points, display in the following case
    \path (-0.8660, 0.5) coordinate (A1);
    \path (-0.8660,-0.5) coordinate (A2);
    \path (0.     ,-1  ) coordinate (A3);
    \path (0.8660 ,-0.5) coordinate (A4);
    \path (0.8660 ,0.5 ) coordinate (A5);
    \path (0.     ,1.0 ) coordinate (A6);

    \path (0.0,0.0) coordinate (center);

    \path (-0.8660,0.   )    coordinate (A1half);
    \path (-0.433 ,-0.75)    coordinate (A2half);
    \path ( 0.433 ,-0.75)    coordinate (A3half);
    \path ( 0.8660,0.   )    coordinate (A4half);
    \path ( 0.433 , 0.75)    coordinate (A5half);
    \path (-0.433 , 0.75)    coordinate (A6half);

    \path (A1half) ++(-0.25 ,  0.25)  coordinate (A1halfe);
    \path (A2half) ++(-0.225, -0.0 )  coordinate (A2halfe);
    \path (A3half) ++( 0.225, -0.0 )  coordinate (A3halfe);
    \path (A4half) ++( 0.25 ,  0.25)  coordinate (A4halfe);
    \path (A5half) ++( 0.225,  0.0 )  coordinate (A5halfe);
    \path (A6half) ++(-0.225,  0.0 )  coordinate (A6halfe);

    \path (A1half) ++(-0.6   ,   0   )  coordinate (A1To);
    \path (A2half) ++(-0.2598,  -0.45)  coordinate (A2To);
    \path (A3half) ++( 0.2598,  -0.45)  coordinate (A3To);
    \path (A4half) ++( 0.6   ,   0   )  coordinate (A4To);
    \path (A5half) ++( 0.2598,   0.45)  coordinate (A5To);
    \path (A6half) ++(-0.2598,   0.45)  coordinate (A6To);
    \draw (A6) -- (A1) -- (A2) -- (A3)--(A4)-- (A5)--(A6);

    \draw node at (center) {$T$};
    \draw node[right] at (A1half) {$M_{1}$};
    \draw node[above] at (A2half) {$M_{2}$};
    \draw node[above] at (A3half) {$M_{3}$};
    \draw node[left]  at (A4half) {$M_{4}$};
    \draw node[below] at (A5half) {$M_{5}$};
    \draw node[below] at (A6half) {$M_{6}$};

    \draw node[right] at (A1halfe) {$e_{1}$};
    \draw node[below] at (A2halfe) {$e_{2}$};
    \draw node[below] at (A3halfe) {$e_{3}$};
    \draw node[left]  at (A4halfe) {$e_{4}$};
    \draw node[above] at (A5halfe) {$e_{5}$};
    \draw node[above] at (A6halfe) {$e_{6}$};

    \draw[->,thick] (A1half) -- (A1To) node[left] {$\mathbf{n}_1$};
    \draw[->,thick] (A2half) -- (A2To) node[below]{$\mathbf{n}_2$};
    \draw[->,thick] (A3half) -- (A3To) node[below]{$\mathbf{n}_3$};
    \draw[->,thick] (A4half) -- (A4To) node[right]{$\mathbf{n}_4$};
    \draw[->,thick] (A5half) -- (A5To) node[above]{$\mathbf{n}_5$};
    \draw[->,thick] (A6half) -- (A6To) node[above]{$\mathbf{n}_6$};

    \draw [fill=red] (A1half) circle (0.05cm);
    \draw [fill=red] (A2half) circle (0.05cm);
    \draw [fill=red] (A3half) circle (0.05cm);
    \draw [fill=red] (A4half) circle (0.05cm);
    \draw [fill=red] (A5half) circle (0.05cm);
    \draw [fill=red] (A6half) circle (0.05cm);

\end{tikzpicture}
\end{center}
\caption{An illustrative hexahedron element}
\label{fig.hexahedron}
\end{figure}
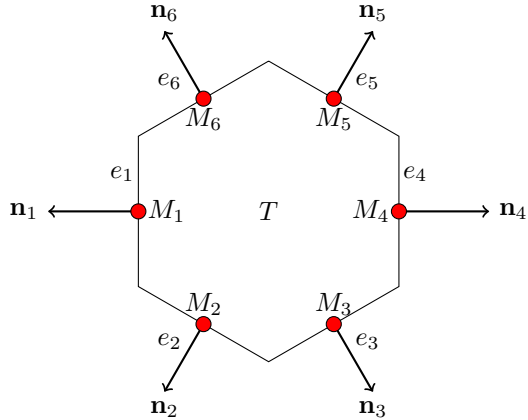

On each element $T\in \T_h$, we introduce the following stabilizer:
\begin{equation}\label{Def-ST-poly}
\begin{split}
S_T(u_b,v_b):=&\ h^{-1}\sum_{i=1}^N (\S(u_b)(M_i)-u_{b,i})(\S(v_b)(M_i)-v_{b,i})|e_i|\\
            = & \ h^{-1}\langle Q_b \S(u_b)-u_{b},Q_b \S(v_b)- v_{b}\rangle_{\partial T}
            \end{split}
\end{equation}
for $u_b, v_b \in W(T)$, where $Q_b$ is the $L^2$ projection operator onto $W(T)$ and $\langle\cdot,\cdot\rangle_\pT$ stands for the usual $L^2$ inner product in $L^2(\pT)$. The weak gradient and the local stabilizer defined in \eqref{DefWGpoly} and \eqref{Def-ST-poly} can be extended naturally to vector-valued function, for example in the two dimensional space we would have
\begin{eqnarray}
&&\nabla_w \bm{u}_b :=\begin{bmatrix} \nabla_w u_b\\ \nabla_w v_b \end{bmatrix},\\
&& S_T(\bm{u}_b,\bm{w}_b) := S_T(u_b,w_b) + S_T(v_b,z_b)
\end{eqnarray}
for $\bm{u}_b = \begin{bmatrix} u_b\\ v_b \end{bmatrix} \in [W(T)]^2$ and $\bm{w}_b = \begin{bmatrix} w_b\\ z_b \end{bmatrix} \in [W(T)]^2$.

For $\bm{u}_b \in [W(T)]^2$, the weak divergence of $\bm{u}_b$ is defined as a constant on $T$ satisfying the following equation:
\begin{equation}\label{equ.weak.divergence}
(\nabla_w \cdot \bm{u}_b,1)_T :=  \langle\bm{u}_b\cdot \bm{n}, 1\rangle_{\partial T}.
\end{equation}
The weak divergence is therefore given by
\begin{equation}\label{equ.weak.divergence.new}
\nabla_w \cdot \bm{u}_b|_T =  \frac{1}{|T|}\sum_{i=1}^N \bm{u}_b|_{e_i}\cdot \bm{n}_i |e_i|.
\end{equation}

The Simplified Weak Galerkin (SWG) method for the Stokes equation \eqref{Equ.Stokes} {\em
seeks $\bm{u}_b \in [W_0(\T_h)]^2$, $p_h \in {P}_0 (\T_h)$ such that}
%\begin{equation}\label{equ.Stokes-FD-SWG-sum}
%\left\{
%\begin{array}{rl}
%\displaystyle\kappa \sum_{T} S_T(\bm{u}_b,\bm{v}_b) + \sum_{T}(\nabla_w \bm{u}_b, \nabla_w \bm{v}_b )_T -\sum_{T}(\nabla_w\cdot \bm{v}_b,p_h)_T
%&=\sum_{T}(\bm{f}, \S(\bm{v}_b))_T,\\ %\qquad \forall\, %\bm{v}_b\in [W_0(\T_h)]^2,\\
%\displaystyle\sum_{T}(\nabla_w\cdot \bm{u}_b,w)_T&=0,%\qquad \forall \, w \in {P}_0(T_h),
%\end{array}
%\right.
%\end{equation}
%or
\begin{equation}\label{equ.Stokes-FD-SWG}
\left\{
\renewcommand\arraystretch{2}
\begin{array}{rl}
\kappa S(\bm{u}_b,\bm{v}_b) + (\nabla_w \bm{u}_b, \nabla_w \bm{v}_b ) -(\nabla_w\cdot \bm{v}_b,p_h)&=(\bm{f}, \S(\bm{v}_b)),\\%\; \forall \; \bm{v}_b\in [W_0(\T_h)]^2,\\
(\nabla_w\cdot \bm{u}_b,w)&=0,% \;\forall \; w \in {P}_0 (T_h),
\end{array}
\right.
\end{equation}
for all $\bm{v}_b\in [W_0(\T_h)]^2 $ and $w \in {P}_0 (T_h)$, where $S(\bm{u}_b,\bm{v}_b):=\sum_{T}S_T(\bm{u}_b,\bm{v}_b)$ and $\kappa>0$ is any prescribed stabilization parameter.

\section{On the Computation of Element Stiffness Matrices}\label{Section:stiffnessmatrix}
The simplified weak Galerkin finite element method \eqref{equ.Stokes-FD-SWG} can be reformulated as follows:
\begin{equation}\label{equ.Stokes-FD-SWG-single}
\left\{
\begin{array}{rl}
\displaystyle \sum_{T} \left( \kappa S_T(u_b,w_b) + (\nabla_w u_b, \nabla_w w_b )_T - (\nabla_w\cdot \begin{bmatrix}w_b\\0 \end{bmatrix} ,p_h)_T\right)&=\displaystyle\sum_{T}( f^{(1)}, \S(w_b))_T,\\
\displaystyle\sum_{T}\left( \kappa S_T(v_b,z_b) + (\nabla_w v_b, \nabla_w z_b )_T -(\nabla_w\cdot \begin{bmatrix}0\\z_b \end{bmatrix},p_h)_T\right)&=\displaystyle\sum_{T}( f^{(2)}, \S(z_b))_T,\\
\displaystyle\sum_{T}(\nabla_w\cdot \begin{bmatrix}u_b\\v_b \end{bmatrix},\chi)_T&=0,
\end{array}
\right.
\end{equation}
for all $w_b, z_b\in W_0(\T_h)$ and $\chi\in P_0(\T_h)$. From \eqref{equ.Stokes-FD-SWG-single}, the element stiffness matrix on $T\in\T_h$ consists of three sub-matrices corresponding to the following forms:
$$
S_T(u_b,w_b),\; (\nabla_w u_b, \nabla_w w_b )_T,\; \text{and }
(\nabla_w\cdot \begin{bmatrix}u_b\\v_b \end{bmatrix},\chi)_T.
$$
The goal of this section is to derive the matrix for each of the three forms.

Denote by
$X_{{u}_b}$, $X_{{u}_b}$, and $P$ the vector representation of $u_b$, $v_b$, and $p_h$ given by
\begin{equation*}
X_{{u}_b}
=
\begin{bmatrix}
u_{b,1} \\
u_{b,2} \\
\vdots  \\
u_{b,N} \\
\end{bmatrix},\;
X_{{v}_b}
=
\begin{bmatrix}
v_{b,1} \\
v_{b,2} \\
\vdots  \\
v_{b,N} \\
\end{bmatrix},\;\text{ and } P=[p_h].
\end{equation*}
Here $u_{b,i}$ and $v_{b,i}$, $i=1\ldots N$, represent the values of $u_b$ and $v_b$ at the midpoint $M_i$ of the edge $e_i$, and $P$ is the value of $p_h$ at the center of $T$. The following theorem gives a computational formula for the element stiffness matrix and the element load vector.

\begin{theorem}\label{THM:ESM}
The element stiffness matrix and the element load vector for the SWG scheme are given in a block matrix form as follows:
\begin{equation}\label{EQ:ESM:01}
\begin{bmatrix}
\kappa h^{-1} A + B & 0 & \; Q_1 \\
0      & \kappa A + B  & \; Q_2\\
Q_1^t & Q_2^t & \; 0 \\
\end{bmatrix}
\begin{bmatrix}
X_{{u}_b} \\
X_{{v}_b} \\
P \\
\end{bmatrix}
\cong \begin{bmatrix}
F_1 \\
F_2 \\
0 \\
\end{bmatrix},
\end{equation}
where the block components in \eqref{EQ:ESM:01} are computed explicitly by:
\begin{eqnarray*}
&&A=\{a_{i,j}\}_{N\times N} = E- EM(M^tEM)^{-1}M^tE, \\
&&B=\{b_{i,j}\}_{N\times N}, \; b_{i,j} = \bm{n}_i\cdot\bm{n}_j\displaystyle\frac{|e_i||e_j|}{|T|},\\
&&D=\{d_{i,j}\}_{3\times N}=(M^tEM)^{-1}M^tE,\\
&&F_1=\{f^{(1)}_{i}\}_{N\times 1}, \; \displaystyle f^{(1)}_{i} = \int_{T}f^{(1)}(x,y) (d_{1,i} + d_{2,i}(x-x_T) + d_{3,i}(y-y_T))dT,\\
&&F_2=\{f^{(2)}_{i}\}_{N\times 1}, \; \displaystyle f^{(2)}_{i} = \int_{T}f^{(2)}(x,y) (d_{1,i} + d_{2,i}(x-x_T) + d_{3,i}(y-y_T)) dT,\\
&&Q_1=\{q_{i}\}_{N\times 1}, \; q_{i} = |e_i|
\begin{bmatrix}
1 \\
0 \\
\end{bmatrix} \cdot \bm{n}_i,\\
&&Q_2=\{q_{i}\}_{N\times 1}, \; q_{i} = |e_i|
\begin{bmatrix}
0 \\
1 \\
\end{bmatrix} \cdot \bm{n}_i.
\end{eqnarray*}
Here $Q_1^t$ and $Q_2^t$ stands for the transpose of $Q_1$ and $Q_2$, respectively. The matrices $M$ and $E$ are given by
\begin{equation*}
M=
\begin{bmatrix}
1      & x_{1} - x_T & y_{1} - y_T\\
1      & x_{2} - x_T & y_{2} - y_T\\
\vdots & \vdots        & \vdots       \\
1      & x_{N} - x_T & y_{N} - y_T\\
\end{bmatrix}_{N\times3},\;
E=
\begin{bmatrix}
|e_1| &       &           &       \\  %\multicolumn{2}{c}{\multirow{2}{*}{\Large 0}}
      & |e_2| &           &       \\
      &       & \ddots    &       \\
      &       &           & |e_N| \\
\end{bmatrix}_{N\times N},
\end{equation*}
where $M_T=(x_T,y_T)$ is any point on the plane (e.g., the center of $T$ as a specific case), $(x_i, y_i)$ is the midpoint of $e_i$, $|e_i|$ is the length of edge $e_i$, $\bm{n}_i$ is the outward normal vector on $e_i$, $|T|$ is the area of the element $T$, and $h$ is the diameter of the element $T$. For simplicity, one may chose $h=\max_i{|e_i|}$.
\end{theorem}

\medskip
The rest of this section is devoted to a derivation of the formula \eqref{EQ:ESM:01}.

\subsection{Element stiffness matrix for $S_T(u_b,w_b)$}
For any $u_b\in W(T)$, the linear extension of $u_b$, $\S(u_b)$, can be represented as follows:
\[
\S(u_b)=\alpha_0 + \alpha_1 (x-x_T) + \alpha_2 (y-y_T).
\]
By \eqref{Def.extension}, we have
\[
\sum_{i=1}^{N}(\S(u_b)(M_i) -u_{b,i})\phi(M_i)|e_i|=0\quad \forall\, \phi\in {P}_1(T).
\]
It follows that
\[
\sum_{i=1}^{N}\left(\alpha_0 + \alpha_1 (x_{i}-x_T) + \alpha_2 (y_{i}-y_T) -  u_{b,i}\right)\phi(M_i)|e_i|=0\;\; \forall\, \phi\in {P}_1(T),
\]
where $(x_i, y_i)$ is the midpoint of $e_i$.
In particular, we may choose $\phi=1, x-x_T, y-y_T$ to obtain
\begin{equation*}
\left\{
\begin{array}{lll}
\displaystyle\sum_{i=1}^{N}\left( \alpha_0 + \alpha_1 (x_{i}-x_T) + \alpha_2 (y_{i}-y_T) -u_{b,i}\right)|e_i|=0,\\
\displaystyle\sum_{i=1}^{N}\left( \alpha_0 + \alpha_1 (x_{i}-x_T) + \alpha_2 (y_{i}-y_T) -u_{b,i}\right) (x_{i}-x_T)|e_i|=0,\\
\displaystyle\sum_{i=1}^{N}\left( \alpha_0 + \alpha_1 (x_{i}-x_T) + \alpha_2 (y_{i}-y_T) -u_{b,i}\right) (y_{i}-y_T)|e_i|=0,
\end{array}
\right.
\end{equation*}
which gives
\begin{equation}\label{equ.poly.1}
\left\{
\begin{array}{lll}
\displaystyle\sum_{i=1}^{N}\left( \alpha_0 + \alpha_1 (x_{i}-x_T) + \alpha_2 (y_{i}-y_T)\right)|e_i| = \sum_{i=1}^{N}u_{b,i}|e_i|,\\
\displaystyle\sum_{i=1}^{N}\left( \alpha_0 + \alpha_1 (x_{i}-x_T) + \alpha_2 (y_{i}-y_T)\right) (x_{i}-x_T)|e_i| = \sum_{i=1}^{N}u_{b,i}(x_{i}-x_T)|e_i|,\\
\displaystyle\sum_{i=1}^{N}\left( \alpha_0 + \alpha_1 (x_{i}-x_T) + \alpha_2 (y_{i}-y_T)\right) (y_{i}-y_T)|e_i| = \sum_{i=1}^{N}u_{b,i}(y_{i}-y_T)|e_i|.
\end{array}
\right.
\end{equation}
Then \eqref{equ.poly.1} can be reformulated as
\begin{equation*}
M^tEM
\begin{bmatrix}
\alpha_0  \\
\alpha_1  \\
\alpha_2  \\
\end{bmatrix}
=M^tE
\begin{bmatrix}
u_{b,1} \\
u_{b,2} \\
\vdots  \\
u_{b,N} \\
\end{bmatrix},
\end{equation*}
which leads to
\begin{equation}\label{EQ:alpha}
\begin{bmatrix}
\alpha_0  \\
\alpha_1  \\
\alpha_2  \\
\end{bmatrix}
=(M^tEM)^{-1}M^tE
\begin{bmatrix}
u_{b,1} \\
u_{b,2} \\
\vdots  \\
u_{b,N} \\
\end{bmatrix}.
\end{equation}
Since $\S(u_b)=\alpha_0 + \alpha_1 (x-x_T) + \alpha_2 (y-y_T)$, we have
\begin{eqnarray*}
&& \begin{bmatrix}
\S(u_{b})(M_1) \\
\S(u_{b})(M_2) \\
\vdots  \\
\S(u_{b})(M_N) \\
\end{bmatrix}
=
\begin{bmatrix}
1      & x_{1} - x_T & y_{1} - y_T\\
1      & x_{2} - x_T & y_{2} - y_T\\
\vdots & \vdots        & \vdots       \\
1      & x_{N} - x_T & y_{N} - y_T\\
\end{bmatrix}
\begin{bmatrix}
\alpha_0\\
\alpha_1\\
\alpha_2\\
\end{bmatrix}
=M
\begin{bmatrix}
\alpha_0\\
\alpha_1\\
\alpha_2\\
\end{bmatrix} \\
&=&M(M^tEM)^{-1}M^tE
\begin{bmatrix}
u_{b,1} \\
u_{b,2} \\
\vdots  \\
u_{b,N} \\
\end{bmatrix}.
\end{eqnarray*}

Let $w_{b}$ be the basis function of $W(T)$ corresponding to the edge $e_j$ of $T$; i.e.,
\begin{equation*}
w_{b}=\left\{
\begin{array}{lllll}
1, \text{ on } e_j,\\
0, \text{ on other edges},\\
\end{array}
\right.
\end{equation*}
that is $w_{b,j}=1$, $w_{b,\{1,\ldots, N\}/j}=0$,
then the coefficient $(\alpha_0,\alpha_1,\alpha_2)^t$ for $\S(w_b)$ is given by
\begin{equation*}
\begin{bmatrix}
\alpha_0  \\
\alpha_1  \\
\alpha_2  \\
\end{bmatrix}
=(M^tEM)^{-1}M^tE
\begin{bmatrix}
w_{b,1}\\
\vdots \\
w_{b,j}\\
\vdots \\
w_{b,N}\\
\end{bmatrix}
=(M^tEM)^{-1}M^tE
\begin{bmatrix}
0\\
\vdots \\
1\\
\vdots \\
0\\
\end{bmatrix}
\triangleq
\begin{bmatrix}
d_{1,j}\\
d_{2,j}\\
d_{3,j}\\
\end{bmatrix}.
\end{equation*}
Thus, we have
\begin{equation}\label{equ.SM.1}
\begin{split}
S_T(u_b,w_b)=&\ \sum_{i=1}^N h^{-1}(\S(u_b)(M_i)-u_{b,i})(\S(w_b)(M_i)-w_{b,i})|e_i|~\\
            =&\ h^{-1}\sum_{i=1}^N (u_{b,i}-\S(u_b)(M_i))w_{b,i}|e_i|\\
            =&\ h^{-1}\left((I_N-M(M^tEM)^{-1}M^tE)
\begin{bmatrix}
u_{b,1} \\
u_{b,2} \\
\vdots  \\
u_{b,N} \\
\end{bmatrix}\right)_j
\cdot w_{b,j}|e_j|
\end{split}
\end{equation}
where $I_N$ is the $N\times N$ identity matrix. The identity \eqref{equ.SM.1} gives rise to the element stiffness matrix for the term $S_T(u_b,w_b)$. The result can be summarized as follows.

\begin{theorem}
For the polygonal element $T$ depicted in Fig. \ref{fig.hexahedron}, the element stiffness matrix corresponding to the bilinear form $S_T(u_b,w_b)$ is given by
\begin{equation}\label{EQ:element-stiffness-matrix-S}
A=\{a_{i,j}\}_{i,j=1}^N: = h^{-1}(E- EM(M^tEM)^{-1}M^tE),
\end{equation}
where
\begin{equation*}
M=
\begin{bmatrix}
1      & x_{1} - x_T & y_{1} - y_T\\
1      & x_{2} - x_T & y_{2} - y_T\\
\vdots & \vdots        & \vdots       \\
1      & x_{N} - x_T & y_{N} - y_T\\
\end{bmatrix}_{N\times3},\;
E=
\begin{bmatrix}
|e_1| &       &           &       \\  %\multicolumn{2}{c}{\multirow{2}{*}{\Large 0}}
      & |e_2| &           &       \\
      &       & \ddots    &       \\
      &       &           & |e_N| \\
\end{bmatrix}_{N\times N},
\end{equation*}
$M_T=(x_T,y_T)$ is the center of $T$ or any point on the plane, $(x_i, y_i)$ is the midpoint of $e_i$, $|e_i|$ is the length of edge $e_i$, $\bm{n}_i$ is the outward normal vector,  $|T|$ is the area of the element T.
\end{theorem}

\subsection{Element stiffness matrix for $(\nabla_w u_b,\nabla_w w_b)_T$} From the weak gradient formulation \eqref{DefWGpoly}, we have
\begin{eqnarray}
(\nabla_w u_b,\nabla_w w_b )_T
&=&(\displaystyle\frac{1}{|T|}\sum_{i=1}^N u_{b,i}\bm{n}_i|e_i|,\displaystyle\frac{1}{|T|}\sum_{j=1}^N w_{b,j}\bm{n}_j|e_j|)_{T}~~\label{equ.SM.2}\\
&=&(\displaystyle\frac{1}{|T|}\sum_{i=1}^N u_{b,i}\bm{n}_i|e_i|,\displaystyle\frac{1}{|T|} w_{b,j}\bm{n}_j|e_j|)_{T} \nonumber\\
&=&\sum_{i=1}^N \frac{|e_i||e_j|}{|T|}(\bm{n}_i\cdot\bm{n}_j) u_{b,i}w_{b,j},\nonumber \\
&=&\sum_{i=1}^N b_{i,j}u_{b,i}w_{b,j}.\nonumber
\end{eqnarray}
Thus, the corresponding element stiffness matrix is given by
\begin{equation}\label{EQ:element-stiffness-matrix-G}
B:=\{b_{i,j}\}_{i,j=1}^N,\; b_{i,j} = (\bm{n}_i\cdot\bm{n}_j)\displaystyle\frac{|e_i||e_j|}{|T|}.
\end{equation}

\subsection{Element stiffness matrix for $(\nabla_w\cdot \bm{u}_b,\chi)_T$} Note that $\chi\in P_0(T)$ is a constant. Thus, we have
\begin{eqnarray}
(\nabla_w\cdot \bm{u}_b,\chi)_T& =& \int_{T} \nabla_w\cdot \bm{u}_b \chi dT = \chi \int_{\partial T} \bm{u}_b \cdot \bm{n} ds  \label{equ.SM.5}\\
&=& \sum_{i=1}^N
\begin{bmatrix}
u_{b,i} \\
v_{b,i} \\
\end{bmatrix}
\cdot \bm{n}_i|e_i| \chi \nonumber\\
&=&
Q_1^t X_{u_b} \chi + Q_2^t X_{v_b} \chi,
\end{eqnarray}
which gives the contribution of $Q_1$ and $Q_2$ in the element stiffness matrix.

\subsection{Element load vector} For a computation of the element load vector corresponding to $(f^{(1)}, \S(w_b))_T$, let $w_{b}$ be the basis function of $W(T)$ corresponding to edge $e_j$ of $T$; i.e.,
\begin{equation*}
w_{b}=\left\{
\begin{array}{lllll}
1, \text{ on } e_j,\\
0, \text{ on other edges}.
\end{array}
\right.
\end{equation*}
The linear extension of $w_b$ is given by
$$
\S(w_b) = d_{1,j} + d_{2,j}(x-x_T) + d_{3,j}(y-y_T),
$$
where
\begin{equation*}
\begin{bmatrix}
d_{1,j}\\
d_{2,j}\\
d_{3,j}\\
\end{bmatrix}
=(M^tEM)^{-1}M^tE
\begin{bmatrix}
w_{b,1}\\
\vdots \\
w_{b,j}\\
\vdots \\
w_{b,N}\\
\end{bmatrix}
=(M^tEM)^{-1}M^tE
\begin{bmatrix}
0\\
\vdots \\
1\\
\vdots \\
0\\
\end{bmatrix}.
\end{equation*}
It follows that
\begin{equation}\label{equ.SM.3}
\begin{split}
(f^{(1)}, \S(w_b))_T
& = \int_{T}f^{(1)} \S(w_b)dT\\
& = \int_{T}f^{(1)}(x,y) (d_{1,j} + d_{2,j}(x-x_T) + d_{3,j}(y-y_T)) dT,
\end{split}
\end{equation}
from which an element load vector can be easily computed. In practical computation, the integral in \eqref{equ.SM.3} should be approximated numerically by some quadrature rules.

\section{SWG on Cartesian Grids}\label{sectionFDSWG}
Consider the Stokes problem (\ref{Equ.Stokes}) defined on the unit square domain $\Omega= (0,1)\times(0,1)$. A uniform mesh of the domain can be given by the Cartesian product of two one-dimensional grids:
\begin{eqnarray*}
x_{i+\mu}&=&(i+\mu-0.5)h,\qquad i=0,1,\cdots, n,\quad \mu=0,1/2,\\
y_{j+\mu}&=&(j+\mu-0.5)h, \qquad j=0,1,\cdots, n,\quad \mu=0,1/2,
\end{eqnarray*}
where $n$ is a positive integer and $h=1/n$ is the meshsize. Denote by
$$
T_{ij}:= [x_{i-\frac12}, x_{i+\frac12}]\times [y_{j-\frac12}, y_{j+\frac12}],\qquad i,j=1,\ldots, n
$$
the square element centered at $(x_i,y_j)$ for $i,j=1,\cdots, n$. The collection of all such elements forms a uniform square partition $\T_h$ of the domain. The collection of all the element edges is denoted as $\E_h$.

The goal of this section is to assemble the global stiffness matrix and the load vector for the SWG scheme associated with uniform rectangular partitions $\T_h$. The resulting matrix problem can be seen as a finite difference scheme for the Stokes equation. In particular, this will result in a new 5-point finite difference scheme when the stabilization parameter has the value of $\kappa=4$.

Let $T\in\T_h$ be a rectangular element depicted in Fig. \ref{fig:rectangular} with center $M_T=(x_T, y_T)$. From \eqref{EQ:alpha}, the linear extension of $u_b \in W(T)$ (i.e., $\S(u_b)$) can be verified to be (see Lemma 6.1 in \cite{LiDanWW} for details):
\begin{equation}\label{EQ:extension}
\S(u_b)=\alpha_0 + \alpha_1(x-x_T)+\alpha_2(y-y_T),
\end{equation}
where
\begin{equation}\label{EQ:S-Property:00222}
\left\{
\begin{array}{lllll}
\alpha_0 =\displaystyle\frac{|e_1|(u_{b,1} + u_{b,2}) + |e_3|(u_{b,3} +u_{b,4}) }{2|e_1|+2|e_3|},\\
\alpha_1 =(u_{b,2} - u_{b,1})/|e_3|,\\
\alpha_2 =(u_{b,4} - u_{b,3})/|e_1|.
\end{array}
\right.
\end{equation}
From \eqref{EQ:extension} we have
\begin{equation}\label{EQ:S-Property:002}
\begin{split}
({\S}(v_b)-v_b)(M_1) & =({\S}(v_b)-v_b)(M_2)\\
& = \frac{|e_3|}{2(|e_1|+|e_3|)}(v_{b3}+v_{b4}-v_{b1}-v_{b2}),\\
\end{split}
\end{equation}
\begin{equation}\label{EQ:S-Property:008}
\begin{split}
({\S}(v_b)-v_b)(M_3) & =({\S}(v_b)-v_b)(M_4)\\
& =- \frac{|e_1|}{2(|e_1|+|e_3|)}(v_{b3}+v_{b4}-v_{b1}-v_{b2}).
\end{split}
\end{equation}
Hence,
\begin{equation}\label{EQ:S-Property:021}
|e_1|({\S}(v_b)-v_b)(M_1) = - |e_3| ({\S}(v_b)-v_b)(M_3).
\end{equation}
A by-product of \eqref{EQ:S-Property:00222} is the following estimate:
\begin{equation}\label{EQ:S-Property:0212}
\|\nabla \S(v)\|_T \leq C \|\nabla_w v\|_T.
\end{equation}
In addition, the linear function $\S(v_b)-v_b$ satisfies
\begin{equation}\label{EQ:July01:001}
\begin{split}
(\S(v_b)-v_b)|_{e_1}(y) = &\ (\S(v_b)-v_b)|_{e_2}(y),\\
(\S(v_b)-v_b)|_{e_3}(x) = &\ (\S(v_b)-v_b)|_{e_4}(x).
\end{split}
\end{equation}

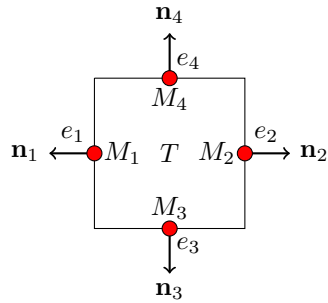
\begin{figure}[!h]
\begin{center}
\begin{tikzpicture}[rotate=0]
    %define the points of triangle
    %you can also use \coordinate to define these points, display in the following case
    \path (0,0) coordinate (A1);
    \path (2,0) coordinate (A2);
    \path (2,2) coordinate (A3);
    \path (0,2) coordinate (A4);

    \path (1.0,1.0) coordinate (center);

    \path (1.0,0) coordinate (A1half);
    \path (2,1)   coordinate (A2half);
    \path (1.0,2) coordinate (A3half);
    \path (0,1)   coordinate (A4half);

    \path (1.25,0) coordinate (A1halfe);
    \path (2,1.25)  coordinate (A2halfe);
    \path (1.25,2) coordinate (A3halfe);
    \path (0,1.25)  coordinate (A4halfe);

    \path (A1half) ++(0, -0.6)  coordinate (A1To);
    \path (A2half) ++(0.6,  0)  coordinate (A2To);
    \path (A3half) ++(0,  0.6)  coordinate (A3To);
    \path (A4half) ++(-0.6, 0)  coordinate (A4To);
    \draw (A4) -- (A1) -- (A2) -- (A3)--(A4);

    \draw node at (center) {$T$};
    \draw node[above] at (A1half) {$M_{3}$};
    \draw node[left]  at (A2half) {$M_{2}$};
    \draw node[below] at (A3half) {$M_{4}$};
    \draw node[right] at (A4half) {$M_{1}$};

    \draw node[below] at (A1halfe) {$e_{3}$};
    \draw node[right]  at (A2halfe) {$e_{2}$};
    \draw node[above] at (A3halfe) {$e_{4}$};
    \draw node[left] at (A4halfe) {$e_{1}$};

    \draw[->,thick] (A1half) -- (A1To) node[below]{$\mathbf{n}_3$};
    \draw[->,thick] (A2half) -- (A2To) node[right]{$\mathbf{n}_2$};
    \draw[->,thick] (A3half) -- (A3To) node[above]{$\mathbf{n}_4$};
    \draw[->,thick] (A4half) -- (A4To) node[left]{$\mathbf{n}_1$};

    \draw [fill=red] (A1half) circle (0.1cm);
    \draw [fill=red] (A2half) circle (0.1cm);
    \draw [fill=red] (A3half) circle (0.1cm);
    \draw [fill=red] (A4half) circle (0.1cm);
\end{tikzpicture}
\end{center}
\caption{An illustrative square element.}
\label{fig:rectangular}
\end{figure}

\subsection{A 7-point finite difference scheme}
Two sets of grid points can be generated by using the 1-d grid points $\{x_i\}$ and $\{y_j\}$. The first consists of the mid-points of all edges in $\E_h$ (i.e., the dotted points colored in red in Figure \ref{fig.Dof_Stokes}), which are used to approximate the velocity field. The second set of grid points consists of all cell centers (i.e., the dotted points colored in blue in Figure \ref{fig.Dof_Stokes}), which are used to approximate the pressure unknown.

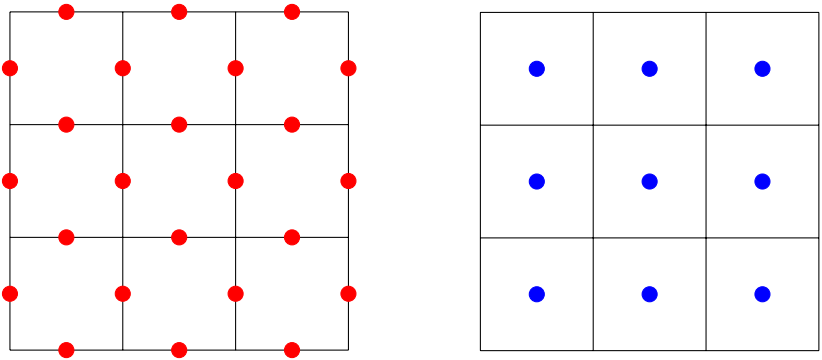
\begin{figure}[!h]
\begin{center}
\subfigure%[Set of grid points for the velocity.]
{
\label{Fig.stokesdof.lu}
\begin{tikzpicture}
    \path (0.75,0) coordinate (A11);
    \path (2.25,0) coordinate (A12);
    \path (3.75,0) coordinate (A13);

    \path (0.75,1.5) coordinate (A21);
    \path (2.25,1.5) coordinate (A22);
    \path (3.75,1.5) coordinate (A23);

    \path (0.75,3) coordinate (A31);
    \path (2.25,3) coordinate (A32);
    \path (3.75,3) coordinate (A33);

    \path (0.75,4.5) coordinate (A41);
    \path (2.25,4.5) coordinate (A42);
    \path (3.75,4.5) coordinate (A43);

    \path (0,0.75)   coordinate (B11);
    \path (1.5,0.75) coordinate (B12);
    \path (3,0.75)   coordinate (B13);
    \path (4.5,0.75) coordinate (B14);

    \path (0,2.25)   coordinate (B21);
    \path (1.5,2.25) coordinate (B22);
    \path (3.0,2.25) coordinate (B23);
    \path (4.5,2.25) coordinate (B24);

    \path (0,3.75)   coordinate (B31);
    \path (1.5,3.75) coordinate (B32);
    \path (3.0,3.75) coordinate (B33);
    \path (4.5,3.75) coordinate (B34);

    \draw[step=1.5] (0,0) grid (4.5,4.5);

    \filldraw [fill=red,draw=red] (A11) circle (0.1cm);
    \filldraw [fill=red,draw=red] (A12) circle (0.1cm);
    \filldraw [fill=red,draw=red] (A13) circle (0.1cm);

    \filldraw [fill=red,draw=red] (A21) circle (0.1cm);
    \filldraw [fill=red,draw=red] (A22) circle (0.1cm);
    \filldraw [fill=red,draw=red] (A23) circle (0.1cm);

    \filldraw [fill=red,draw=red] (A31) circle (0.1cm);
    \filldraw [fill=red,draw=red] (A32) circle (0.1cm);
    \filldraw [fill=red,draw=red] (A33) circle (0.1cm);

    \filldraw [fill=red,draw=red] (A41) circle (0.1cm);
    \filldraw [fill=red,draw=red] (A42) circle (0.1cm);
    \filldraw [fill=red,draw=red] (A43) circle (0.1cm);

    \filldraw [fill=red,draw=red] (B11) circle (0.1cm);
    \filldraw [fill=red,draw=red] (B12) circle (0.1cm);
    \filldraw [fill=red,draw=red] (B13) circle (0.1cm);
    \filldraw [fill=red,draw=red] (B14) circle (0.1cm);

    \filldraw [fill=red,draw=red] (B21) circle (0.1cm);
    \filldraw [fill=red,draw=red] (B22) circle (0.1cm);
    \filldraw [fill=red,draw=red] (B23) circle (0.1cm);
    \filldraw [fill=red,draw=red] (B24) circle (0.1cm);

    \filldraw [fill=red,draw=red] (B31) circle (0.1cm);
    \filldraw [fill=red,draw=red] (B32) circle (0.1cm);
    \filldraw [fill=red,draw=red] (B33) circle (0.1cm);
    \filldraw [fill=red,draw=red] (B34) circle (0.1cm);
\end{tikzpicture}
}\qquad\qquad
\subfigure%[Set of grid points for the pressure.]
{
\label{Fig.stokesdof.lp}
\begin{tikzpicture}
    \path (0.75,0.75) coordinate (A11);
    \path (2.25,0.75) coordinate (A12);
    \path (3.75,0.75) coordinate (A13);

    \path (0.75,2.25) coordinate (A21);
    \path (2.25,2.25) coordinate (A22);
    \path (3.75,2.25) coordinate (A23);

    \path (0.75,3.75) coordinate (A31);
    \path (2.25,3.75) coordinate (A32);
    \path (3.75,3.75) coordinate (A33);

    \path (3.75,-0.1) coordinate (A14);
    \path (3.75,4.6) coordinate (A15);
    \draw[step=1.5] (0,0) grid (4.5,4.5);

    \filldraw [fill=blue,draw=blue] (A11) circle (0.1cm);
    \filldraw [fill=blue,draw=blue] (A12) circle (0.1cm);
    \filldraw [fill=blue,draw=blue] (A13) circle (0.1cm);

    \filldraw [fill=blue,draw=blue] (A21) circle (0.1cm);
    \filldraw [fill=blue,draw=blue] (A22) circle (0.1cm);
    \filldraw [fill=blue,draw=blue] (A23) circle (0.1cm);

    \filldraw [fill=blue,draw=blue] (A31) circle (0.1cm);
    \filldraw [fill=blue,draw=blue] (A32) circle (0.1cm);
    \filldraw [fill=blue,draw=blue] (A33) circle (0.1cm);
\end{tikzpicture}
}
\end{center}
\caption{Grid points for the Stokes equation: (a) for the velocity (left, red dots), and (b) for the pressure (right, blue dots).}
\label{fig.Dof_Stokes}
\end{figure}

Let $\bm{u}_{k\ell}=(u_{k\ell}, v_{k\ell})^t$ be the velocity approximation at the red-dotted grid points $(x_k, y_\ell)$ and $p_{ij}$ be the approximate pressure at the blue-dotted grid points $(x_i,y_j)$. It can be seen that a red-dotted grid point $(x_k, y_\ell)$ is located on the boundary if either $k$ or $\ell$ takes the value of $\frac12$ or $n+\frac12$. The following is the main result of this section.

\begin{theorem}\label{DF-equi-SWG-new}
The simplified weak Galerkin scheme \eqref{equ.Stokes-FD-SWG} is algebraically equivalent to the following finite difference scheme on uniform square partitions:
\begin{equation}\label{equ.Stokes_FD-NEW}
\left\{
\begin{split}
&{c_1 \bu_{i+\frac{3}{2},j} + c_2 \bu_{i+\frac{1}{2},j} + c_3 \bu_{i-\frac{1}{2},j} + c_4[\bu_{i+1,j-\frac{1}{2}} + \bu_{i+1,j+\frac{1}{2}} + \bu_{i,j-\frac{1}{2}}+ \bu_{i,j+\frac{1}{2}}]} \\
& \qquad \qquad  +
h\begin{bmatrix}
  {p_{i+1,j}-p_{i,j}} \\
  0
\end{bmatrix}
=\frac{h^2}{2}\bbf_{i+\frac{1}{2},j},\\
&
{c_1 \bu_{i,j+\frac{3}{2}} + c_2\bu_{i,j+\frac{1}{2}} + c_3\bu_{i,j-\frac{1}{2}} + c_4[\bu_{i-\frac{1}{2},j+1} + \bu_{i+\frac{1}{2},j+1} +\bu_{i-\frac{1}{2},j}+\bu_{i+\frac{1}{2},j}]}\\
& \qquad \qquad  + h
\begin{bmatrix}
  0 \\
  {p_{i,j+1}-p_{i,j}}
\end{bmatrix}
=\frac{h^2}{2}\bbf_{i,j+\frac{1}{2}}\\
&h[u_{i+\frac{1}{2},j} - u_{i-\frac{1}{2},j}] + h[v_{i,j+\frac{1}{2}} - v_{i,j-\frac{1}{2}}]=0.
\end{split}
\right.
\end{equation}
where
$ c_1=c_3=\displaystyle(\frac{\kappa}{4}-1),\; c_2=\displaystyle\frac{\kappa}{2}+2,\; c_4=\displaystyle-\frac{\kappa}{4}$, $\kappa>0$ is the stabilization parameter. For $\kappa=4$, we have $c_1=c_3=0$, $c_2=4$ and $c_4=-1$ so that the scheme \eqref{equ.Stokes_FD-NEW} gives rise to the five-point finite difference scheme \eqref{equ.Stokes_FDS}.
\end{theorem}

The rest of this subsection is devoted to a derivation of the finite difference scheme  \eqref{equ.Stokes_FD-NEW}. To this end, let $w_b$ be the basis function of $W(T)$ corresponding to edge $e_1$ of $T$ (see Fig. \ref{fig:rectangular}) so that
\begin{equation*}
w_b=\left\{
\begin{array}{lllll}
1,\; \text{ on } e_1,\\
0,\; \text{ on } e_i, \, i=2,3,4.\\
\end{array}
\right.
\end{equation*}
On square elements where $|e_1|=|e_3|$, from \eqref{EQ:S-Property:008} it is not hard to see that
\begin{eqnarray*}
(w_b-\S(w_b))|_{M_1} &=&(w_b-\S(w_b))|_{M_2} =\displaystyle\frac{1}{4},\\
(w_b-\S(w_b))|_{M_3} &=&(w_b-\S(w_b))|_{M_4} =-\displaystyle\frac{1}{4}.
\end{eqnarray*}
Thus, we have
\begin{eqnarray}
\langle u_b- \S(u_b), w_b-\S(w_b)\rangle_{\partial T}&=&\langle u_b, w_b-\S(w_b)\rangle_{\partial T}\label{equ.SM.square.1}\\
&=&\sum_{i=1}^4 |e_i|u_{b,i}(w_b-\S(w_b))|_{M_i}\nonumber\\
&=&\displaystyle\frac{h}{4}(u_{b,1}+u_{b,2} -u_{b,3}-u_{b,4}),\nonumber
\end{eqnarray}
\begin{eqnarray}
(\nabla_w u_b, \nabla_w w_b)_{T} &= &\left(\displaystyle\frac{1}{|T|}\sum_{i=1}^4 u_{b,i}|e_i|\bm{n}_i,\displaystyle\frac{1}{|T|}\sum_{i=1}^4 w_{b,i}|e_i|\bm{n}_i\right)\label{equ.SM.square.2}\\
&=&\left(
\displaystyle\frac{1}{h}\left[
\begin{array}{lllll}
u_{b,2} - u_{b,1} \\
u_{b,4} - u_{b,3} \\
\end{array}
\right],
\displaystyle\frac{1}{h}\left[
\begin{array}{lllll}
0 - 1\\
0 - 0\\
\end{array}
\right]
\right)_T \nonumber\\
&=&-|T|\left(\frac{1}{h}\right)^2(u_{b,2}-u_{b,1})\nonumber\\&=&u_{b,1}-u_{b,2},\nonumber
\end{eqnarray}
\begin{eqnarray}
(f^{(1)}, \S(w_b))_T  &=&\displaystyle\int_T f^{(1)} \left(\frac{1}{4}-\frac{1}{h}(x-x_T)\right)dT~~~~~~~~~~~~~~~~~~ \label{equ.SM.square.3}\\
                 &=&\displaystyle \frac{h^2}{6} \left[\frac{3}{4}f^{(1)}|_{M_1} +f^{(1)}|_{M_T} - \frac{1}{4}f^{(1)}|_{M_2} \right] + \O(h^4).\nonumber
\end{eqnarray}
Analogously, the same techniques can be applied to the computation of $S_T(v_b,w_b)$, $(\nabla_w v_b,\nabla_w w_b )_T$, and $(f^{(2)}, \S(w_b))_T$.

Next, let $\bm{v}_b$ be the vector basis function of $[W(T)]^2$ associated with edge $e_1$, that is
\begin{equation*}
\bm{v}_b =
\begin{bmatrix}
w_b\\
0 \\
\end{bmatrix},\,
\text{ or }
\bm{v}_b =
\begin{bmatrix}
0 \\
w_b \\
\end{bmatrix}, \text{ with }
w_b=\left\{
\begin{array}{lllll}
1, \text{ on } e_1,\\
0, \text{ on other edges}.\\
\end{array}
\right.
\end{equation*}
Then we have
\begin{eqnarray}
(\nabla_w\cdot \bm{v}_b,p_h)_T
&=& p_h |e_1|
\begin{bmatrix}
1 \\
0 \\
\end{bmatrix} \cdot \bm{n}_1,\;
\text{ or }
= p_h |e_1|
\begin{bmatrix}
0 \\
1 \\
\end{bmatrix} \cdot \bm{n}_1 \label{equ.SM.square.4}\\
&=& -hp_h,\; \text{ or } =0.\nonumber
\end{eqnarray}
\begin{eqnarray}
(\nabla_w\cdot \bm{u}_b,\chi)_T &=&
\sum_{i=1}^N u_{b,i}
\begin{bmatrix}
1 \\
0 \\
\end{bmatrix}\cdot \bm{n}_i|e_i|
+
\sum_{i=1}^N v_{b,i}
\begin{bmatrix}
0 \\
1 \\
\end{bmatrix}\cdot \bm{n}_i|e_i| \label{equ.SM.square.5}\\
&=&h(u_{b,2}-u_{b,1}) + h(v_{b,4}-v_{b,3}).\nonumber
\end{eqnarray}
Here $\chi$ is the characteristic function of the element $T$.
Equations \eqref{equ.SM.square.1}, \eqref{equ.SM.square.2}, \eqref{equ.SM.square.3}, \eqref{equ.SM.square.4} and \eqref{equ.SM.square.5} comprise the discrete scheme corresponding to the basis function $w_b$ on edge $e_1$ of the element $T$:
\begin{equation}\label{ElementStencilstokes}
\left\{
\renewcommand\arraystretch{2}
\begin{array}{lllll}
\displaystyle \frac{\kappa}{4}(u_{b,1}+u_{b,2} -u_{b,3}-u_{b,4})+ u_{b,1}-u_{b,2} -h p_h = (f^{(1)}, \S(w_b))_T \\
\displaystyle \frac{\kappa}{4}(v_{b,1}+v_{b,2} -v_{b,3}-v_{b,4})+ v_{b,1}-v_{b,2}  = (f^{(2)}, \S(w_b))_T\\
h(u_{b,2}-u_{b,1}) + h(v_{b,4}-v_{b,3})=0
\end{array}
\right.
\end{equation}

\begin{figure}[!h]
\begin{center}
\subfigure[Stencil for $u_{i+\frac{1}{2},j}$]{
\label{Fig.sub.ldofswgx}
\tikzset{global scale/.style={
         scale=#1,
         every node/.append style ={scale=#1}
         }
}
\begin{tikzpicture}[global scale =0.75]
\draw[step=3] (0,0) grid (6,3);

\path (0,1.5) coordinate (A11);
\path (3,1.5) coordinate (A12);
\path (6,1.5) coordinate (A13);

\path (1.5,0) coordinate (B11);
\path (4.5,0) coordinate (B12);

\path (1.5,3) coordinate (B21);
\path (4.5,3) coordinate (B22);

    \draw [fill=red] (A11) circle (0.1cm);
    \draw [fill=blue] (A12) circle (0.15cm);
    \draw [fill=red] (A13) circle (0.1cm);

    \draw [fill=red] (B11) circle (0.1cm);
    \draw [fill=red] (B12) circle (0.1cm);

    \draw [fill=red] (B21) circle (0.1cm);
    \draw [fill=red] (B22) circle (0.1cm);

    \path (1.5,1.5) coordinate (center1);
    \path (4.5,1.5) coordinate (center2);

    \draw node[left] at (A11) {$i-\frac{1}{2},j$};
    %\draw node[below] at (A12) {$i+\frac{1}{2},j$};
    \draw node[above] at (A12) {$u_{i+\frac{1}{2},j}$};
    \draw node[right] at (A13) {$i+\frac{3}{2},j$};

    \draw node[below] at (B11) {$i,j-\frac{1}{2}$};
    \draw node[below] at (B12) {$i+1,j-\frac{1}{2}$};
    \draw node[above] at (B21) {$i,j+\frac{1}{2}$};
    \draw node[above] at (B22) {$i+1,j+\frac{1}{2}$};

    \draw node at (center1) {$T_{i,j}$};
    \draw node at (center2) {$T_{i+1,j}$};

\end{tikzpicture}
}
\subfigure[Stencil for $u_{i,j+\frac{1}{2}}$]{
\label{Fig.sub.ldofswgy}
\tikzset{global scale/.style={
         scale=#1,
         every node/.append style ={scale=#1}
         }
}
\begin{tikzpicture}[global scale =0.75]
\draw[step=3] (0,0) grid (3,6);

\path (1.5,0) coordinate (A11);
\path (1.5,3) coordinate (A12);
\path (1.5,6) coordinate (A13);

\path (0,1.5) coordinate (B11);
\path (0,4.5) coordinate (B12);

\path (3,1.5) coordinate (B21);
\path (3,4.5) coordinate (B22);

    \draw [fill=red] (A11) circle (0.1cm);
    \draw [fill=blue] (A12) circle (0.15cm);
    \draw [fill=red] (A13) circle (0.1cm);

    \draw [fill=red] (B11) circle (0.1cm);
    \draw [fill=red] (B12) circle (0.1cm);

    \draw [fill=red] (B21) circle (0.1cm);
    \draw [fill=red] (B22) circle (0.1cm);

    \path (1.5,1.5) coordinate (center1);
    \path (1.5,4.5) coordinate (center2);

    \draw node[below]  at (A11) {$i,j-\frac{1}{2}$};
    %\draw node[below]  at (A12) {$i,j+\frac{1}{2}$};
    \draw node[above]  at (A12) {$u_{i,j+\frac{1}{2}}$};
    \draw node[above]  at (A13) {$i,j+\frac{3}{2}$};

    \draw node[left] at (B11)  {$i-\frac{1}{2},j$};
    \draw node[left] at (B12)  {$i-\frac{1}{2},j+1$};
    \draw node[right] at (B21) {$i+\frac{1}{2},j$};
    \draw node[right] at (B22) {$i+\frac{1}{2},j+1$};

    \draw node at (center1) {$T_{i,j}$};
    \draw node at (center2) {$T_{i,j+1}$};

\end{tikzpicture}
}
\end{center}
\caption{\label{Fig.DF.stencils}Stencils for the finite difference scheme \eqref{equ.Stokes_FD-NEW}.}
\end{figure}

The two local equations \eqref{ElementStencilstokes} corresponding to the elements $T_{i,j}$ and $T_{i+1,j}$ that share $e_{i+\frac{1}{2},j}$ as a common edge (see Fig. \ref{Fig.sub.ldofswgx}) are given by
\begin{equation}\label{equ.ES.1}
\begin{split}
& \frac{\kappa}{4}(u_{i+\frac{1}{2},j}+u_{i+\frac{3}{2},j} -u_{i+1,j-\frac{1}{2}}-u_{i+1,j+\frac{1}{2}})+ u_{i+\frac{1}{2},j}-u_{i+\frac{3}{2},j} +h p_{i+1,j}\\
=&\int_{T_{i+1,j}} f^{(1)} \S(w_b)dx
\end{split}
\end{equation}
and
\begin{equation}\label{equ.ES.2}
\begin{split}
&\frac{\kappa}{4}(u_{i+\frac{1}{2},j}+u_{i-\frac{1}{2},j} -u_{i,j-\frac{1}{2}}-u_{i,j+\frac{1}{2}})+ u_{i+\frac{1}{2},j}-u_{i-\frac{1}{2},j} -h p_{i,j}\\
=&\int_{T_{i,j}} f^{(1)} \S(w_b)dx.
\end{split}
\end{equation}
Summing up the equations (\ref{equ.ES.1}) and (\ref{equ.ES.2}) yields the following global linear equation corresponding to the degree of freedom $u_{i+\frac{1}{2},j}$:
\begin{equation}\label{equ.DFSWG.1}
\begin{split}
& c_1u_{i+\frac{3}{2},j}+ c_2u_{i+\frac{1}{2},j}+ c_3u_{i-\frac{1}{2},j} +c_4[u_{i+1,j-\frac{1}{2}}+u_{i+1,j+\frac{1}{2}}+u_{i,j-\frac{1}{2}}
+u_{i,j+\frac{1}{2}}] \\
& \quad +h(p_{i+1,j}-p_{i,j})\\
 = & \int_{T_{i,j}\bigcup T_{i+1,j}} f^{(1)} \S(w_b)dT,
\end{split}
\end{equation}
where $\displaystyle c_1=c_3 = \frac{\kappa}{4}-1$, $\displaystyle c_2 = \frac{\kappa}{2}+2$, $\displaystyle c_4=-\frac{\kappa}{4}$.

The right-hand side of \eqref{equ.DFSWG.1} can be approximated by using proper numerical integrations (the Simpson rule in the $x$- direction and the midpoint rule in the $y$- direction) as follows:
\begin{eqnarray}
&&\int_{T_{i,j}\bigcup T_{i+1,j}} f^{(1)} \S(w_b)dT \nonumber\\ &=&\int_{T_{i,j}}f^{(1)}\S(w_{b,(i+\frac{1}{2},j)})|_{T_{i,j}} dT+ \int_{T_{i+1,j}}f^{(1)}\S(w_{b,(i+\frac{1}{2},j)})|_{T_{i+1,j}}dT \nonumber\\
&=&\frac{h^2}{6}[f^{(1)} \S(w_{b})|_{M_{i-\frac{1}{2},j}} + 4f^{(1)}\S(w_{b})|_{M_{i,j}} + f^{(1)}\S(w_{b})|_{M_{i+\frac{1}{2},j}}]\nonumber\\
&&+\frac{h^2}{6}[f^{(1)}\S(w_{b})|_{M_{i+\frac{1}{2},j}} + 4f^{(1)}\S(w_{b})|_{M_{i+1,j}} + f^{(1)}\S(w_{b})|_{M_{i+\frac{3}{2},j}}]+\O(h^4) \label{equ.Simpson}\\
&=& \frac{h^2}{6}[-\frac{1}{4}f^{(1)}_{i-\frac{1}{2},j} + f^{(1)}_{i,j} + \frac{3}{4}f^{(1)}_{i+\frac{1}{2},j}]+\frac{h^2}{6}[\frac{3}{4}
f^{(1)}_{i+\frac{1}{2},j} + f^{(1)}_{i+1,j} - \frac{1}{4}f^{(1)}_{i+\frac{3}{2},j}]+\O(h^4)\nonumber\\
&=&\frac{h^2}{24}[-f^{(1)}_{i-\frac{1}{2},j} + 4f^{(1)}_{i,j} +6f^{(1)}_{i+\frac{1}{2},j} +4f^{(1)}_{i+1,j}-f^{(1)}_{i+\frac{3}{2},j}] +\O(h^4) \nonumber \\%\label{equ.approx.Simpson.1}\\
&= &\frac{h^2}{2}f^{(1)}_{i+\frac{1}{2},j} + \O(h^4), \nonumber %\label{equ.approx.Simpson.2}
\end{eqnarray}
where we have used the fact that
\begin{eqnarray*}
\S(w_{b,(i+\frac{1}{2},j)})|_{T_{i,j}} &=&\frac{1}{4} + \frac{1}{h} (x-x_{T_{i,j}}),\\
\S(w_{b,(i+\frac{1}{2},j)})|_{T_{i+1,j}} &=&\frac{1}{4} - \frac{1}{h}(x-x_{T_{i+1,j}}).
\end{eqnarray*}
It follows that \eqref{equ.DFSWG.1} can be rewritten as
\begin{equation}\label{equ.DFSWG.1.new}
\begin{split}
& c_1u_{i+\frac{3}{2},j}+ c_2u_{i+\frac{1}{2},j}+ c_3u_{i-\frac{1}{2},j} +c_4[u_{i+1,j-\frac{1}{2}}+u_{i+1,j+\frac{1}{2}}+u_{i,j-\frac{1}{2}}
+u_{i,j+\frac{1}{2}}] \\
& \quad +h(p_{i+1,j}-p_{i,j})\\
 = & \frac{h^2}{2}f^{(1)}_{i+\frac{1}{2},j} + \O(h^4).
\end{split}
\end{equation}
Analogously, for the degree of freedom $u_{i,j+\frac{1}{2}}$, we may obtain
\begin{equation}\label{equ.DFSWG.2}
 \begin{split}
 &c_1u_{i,j+\frac{3}{2}} + c_2u_{i,j+\frac{1}{2}} + c_3u_{i,j-\frac{1}{2}} + c_4[u_{i-\frac{1}{2},j+1}+u_{i+\frac{1}{2},j+1} +u_{i-\frac{1}{2},j}+u_{i+\frac{1}{2},j})] \\
 &\quad =\frac{h^2}{2}f^{(1)}_{i,j+\frac{1}{2}} + \O(h^4).
\end{split}
\end{equation}
The stencil for the unknown $u$, or more precisely for $u_{i+\frac{1}{2},j}$ (respectively $u_{i,j+\frac{1}{2}}$ ) is the seven dotted-points shown in Fig. \ref{Fig.DF.stencils}(a) (respectively Fig. \ref{Fig.DF.stencils}(b)) with weights
\begin{equation}\label{EQ:weight-7points}
A=(\displaystyle\frac{\kappa}{4}-1,\displaystyle\frac{\kappa}{2}+2,
\displaystyle\frac{\kappa}{4}-1,-\displaystyle\frac{\kappa}{4},
-\displaystyle\frac{\kappa}{4},-\displaystyle\frac{\kappa}{4},
-\displaystyle\frac{\kappa}{4})
\end{equation}
for $\kappa>0$.

The same calculation can be carried out for the second component $v$ of the velocity variable; details are omitted.
%\end{proof}

\subsection{A 5-point finite difference Scheme}\label{subsection-5point}

For the particular value of $\kappa=4$, the weight $A$ in \eqref{EQ:weight-7points} for the 7-point stencil becomes to be
\begin{equation}\label{EQ:weight-7points-s}
A=(0, 4, 0, -1, -1, -1, -1)
\end{equation}
so that \eqref{equ.Stokes_FD-NEW} is reduced to a 5-point finite difference scheme described as follows:

\begin{FD-algorithm}
Find $\bm{u}_{k\ell}$ and $p_{ij}$ such that (1) the discrete homogeneous Dirichlet boundary condition of $\bm{u}_{k\ell}=0$ is satisfied at all the red dotted grid points $(x_k, y_\ell)$ on the domain boundary, and (2) the following set of finite difference equations are satisfied:

\setlength{\arraycolsep}{0.5pt}
\begin{equation}\label{equ.Stokes_FDS}
\left\{
\begin{split}
\frac{4\bu_{i+\frac{1}{2},j}-\bu_{i+1,j-\frac{1}{2}} - \bu_{i+1,j+\frac{1}{2}} - \bu_{i,j-\frac{1}{2}}- \bu_{i,j+\frac{1}{2}}}{h^2} +
\begin{bmatrix}
  \frac{p_{i+1,j}-p_{i,j}}{h} \\
  0
\end{bmatrix}
&=\frac{1}{2}\bbf_{i+\frac{1}{2},j},\\
\frac{4\bu_{i,j+\frac{1}{2}}-\bu_{i-\frac{1}{2},j+1} - \bu_{i+\frac{1}{2},j+1} -\bu_{i-\frac{1}{2},j}-\bu_{i+\frac{1}{2},j}}{h^2}+
 \begin{bmatrix}
 0\\
  \frac{p_{i,j+1}-p_{i,j}}{h}
\end{bmatrix}
 &=\frac12\bbf_{i, j+\frac{1}{2}},\\
 u_{i+\frac{1}{2},j} - u_{i-\frac{1}{2},j} + v_{i,j+\frac{1}{2}} - v_{i,j-\frac{1}{2}}&=0,\\
\end{split}
\right.
\end{equation}
where
$$
\bbf_{k,\ell}=
\begin{bmatrix}
  f^{(1)}_{k,\ell}\\
  f^{(2)}_{k,\ell}
\end{bmatrix}
:=
\begin{bmatrix}
  f^{(1)}(x_k,y_\ell)\\
  f^{(2)}(x_k,y_\ell)
\end{bmatrix}.
$$
\end{FD-algorithm}
\medskip

By using the component notation for the velocity $\bm{u}=(u,v)^t$, we may rewrite the finite difference equations \eqref{equ.Stokes_FDS} as follows:

\setlength{\arraycolsep}{0.5pt}
\begin{equation}\label{equ.Stokes_FDS_comp}
\left\{
\renewcommand\arraystretch{2}
\begin{array}{rl}
 \displaystyle(4u_{i+\frac{1}{2},j}-u_{i+1,j-\frac{1}{2}} - u_{i+1,j+\frac{1}{2}} - u_{i,j-\frac{1}{2}}- u_{i,j+\frac{1}{2}}) + h(p_{i+1,j}-p_{i,j})&=\displaystyle\frac{h^2}{2}f^{(1)}_{i+\frac{1}{2},j},\\
\displaystyle(4v_{i+\frac{1}{2},j}-v_{i+1,j-\frac{1}{2}} - v_{i+1,j+\frac{1}{2}} -v_{i,j-\frac{1}{2}}-v_{i,j+\frac{1}{2}}) &=\displaystyle\frac{h^2}{2}f^{(2)}_{i+\frac{1}{2},j},\\
 \displaystyle(4u_{i,j+\frac{1}{2}}-u_{i-\frac{1}{2},j+1} - u_{i+\frac{1}{2},j+1} -u_{i-\frac{1}{2},j}-u_{i+\frac{1}{2},j})&=
 \displaystyle\frac{h^2}{2}f^{(1)}_{i,j+\frac{1}{2}},\\
 \displaystyle(4v_{i,j+\frac{1}{2}}-v_{i-\frac{1}{2},j+1} - v_{i+\frac{1}{2},j+1} - v_{i-\frac{1}{2},j}- v_{i+\frac{1}{2},j}) + h(p_{i,j+1}-p_{i,j})&=\displaystyle\frac{h^2}{2}f^{(2)}_{i,j+\frac{1}{2}},\\
 u_{i+\frac{1}{2},j} - u_{i-\frac{1}{2},j} + v_{i,j+\frac{1}{2}} - v_{i,j-\frac{1}{2}}&=0.\\
\end{array}
\right.
\end{equation}

\begin{figure}[!h]
\begin{center}
\subfigure[Stencil for $\bm{u}_{i+\frac{1}{2},j}$]{
\label{Fig.sub.2.ldofswgx}
\tikzset{global scale/.style={
         scale=#1,
         every node/.append style ={scale=#1}
         }
}
\begin{tikzpicture}[global scale =0.66]
\draw[step=3] (0,0) grid (6,3);
\path (0,1.5) coordinate (A11);
\path (3,1.5) coordinate (A12);
\path (6,1.5) coordinate (A13);

\path (1.5,0) coordinate (B11);
\path (4.5,0) coordinate (B12);

\path (1.5,3) coordinate (B21);
\path (4.5,3) coordinate (B22);

    %\draw [fill=red] (A11) circle (0.1cm);
    \draw [fill=red] (A12) circle(0.15cm);
    %\draw [fill=red] (A13) circle (0.1cm);

    \draw [fill=red] (B11) circle (0.15cm);
    \draw [fill=red] (B12) circle (0.15cm);

    \draw [fill=red] (B21) circle (0.15cm);
    \draw [fill=red] (B22) circle (0.15cm);

    \path (1.5,1.5) coordinate (center1);
    \path (4.5,1.5) coordinate (center2);

    %\draw node[left] at (A11) {$i-\frac{1}{2},j$};
    \draw node[above] at (A12) {$\bm{u}_{i+\frac{1}{2},j}$};
    %\draw node[right] at (A13) {$i+\frac{3}{2},j$};

    \draw node[below] at (B11) {$i,j-\frac{1}{2}$};
    \draw node[below] at (B12) {$i+1,j-\frac{1}{2}$};
    \draw node[above] at (B21) {$i,j+\frac{1}{2}$};
    \draw node[above] at (B22) {$i+1,j+\frac{1}{2}$};

    \draw node at (center1) {$T_{i,j}$};
    \draw node at (center2) {$T_{i+1,j}$};

\end{tikzpicture}
}
~~~~~~
\subfigure[Stencil for $\bm{u}_{i,j+\frac{1}{2}}$]{
\label{Fig.sub.2.ldofswgy}
\tikzset{global scale/.style={
         scale=#1,
         every node/.append style ={scale=#1}
         }
}
\begin{tikzpicture}[global scale =0.66]
\draw[step=3] (0,0) grid (3,6);

\path (1.5,0) coordinate (A11);
\path (1.5,3) coordinate (A12);
\path (1.5,6) coordinate (A13);

\path (0,1.5) coordinate (B11);
\path (0,4.5) coordinate (B12);

\path (3,1.5) coordinate (B21);
\path (3,4.5) coordinate (B22);

    %\draw [fill=red] (A11) circle (0.1cm);
    \draw [fill=red] (A12) circle(0.15cm);
    %\draw [fill=red] (A13) circle (0.1cm);

    \draw [fill=red] (B11) circle (0.15cm);
    \draw [fill=red] (B12) circle (0.15cm);

    \draw [fill=red] (B21) circle (0.15cm);
    \draw [fill=red] (B22) circle (0.15cm);

    \path (1.5,1.5) coordinate (center1);
    \path (1.5,4.5) coordinate (center2);

    %\draw node[below]  at (A11) {$i,j-\frac{1}{2}$};
    \draw node[above]  at (A12) {$\bm{u}_{i,j+\frac{1}{2}}$};
    %\draw node[above]  at (A13) {$i,j+\frac{3}{2}$};

    \draw node[left] at (B11)  {$i-\frac{1}{2},j$};
    \draw node[left] at (B12)  {$i-\frac{1}{2},j+1$};
    \draw node[right] at (B21) {$i+\frac{1}{2},j$};
    \draw node[right] at (B22) {$i+\frac{1}{2},j+1$};

    \draw node at (center1) {$T_{i,j}$};
    \draw node at (center2) {$T_{i,j+1}$};

\end{tikzpicture}
}
\end{center}
\caption{\label{Fig.DF.stencil2}Stencils for the velocity field in the finite difference scheme \eqref{equ.Stokes_FDS}.}
\end{figure}

Figure \ref{Fig.DF.stencil2} shows the stencil of  \eqref{equ.Stokes_FDS} for the velocity field as a five-point finite difference scheme with weights $\tilde A=(4,-1,-1,-1,-1)$; the left figure is for the scheme centered at $(x_{i+\frac12,j}, y_j)$ and the right one is for $(x_{i}, y_{j+\frac12})$.

\section{Stability}\label{Section:Stability}
For simplicity, we introduce two bilinear forms:
\begin{eqnarray}
a(\bm{u},\bm{v}) &:=& \kappa S(\bm{u},\bm{v}) + (\nabla_w \bm{u}, \nabla_w \bm{v}),\label{EQ:June30:001}\\
b(\bm{v},w) &:=& (\nabla_w \cdot \bm{v},w),\label{EQ:June30:002}
\end{eqnarray}
where $\bm{u}, \bm{v} \in [W_0(\T_h)]^2$ and $w \in {P}_0 (\T_h)$.
The SWG scheme \eqref{equ.Stokes-FD-SWG} is a typical saddle-point problem which can be analyzed by using the well known theory developed by Babu\u{s}ka \cite{Babuska1973} and Brezzi \cite{Brezzi1974}. The core of the theory is to verify two properties:
\begin{itemize}
\item[(i)] boundedness and a certain coercivity for the bilinear form $a(\cdot,\cdot)$,
\item[(ii)] boundedness and inf-sup condition for the bilinear form $b(\cdot,\cdot)$,
\end{itemize}
which are to be established in the rest of this section.

The space $[W_0(\T_h)]^2$ is a normed linear space equipped with the following triple-bar norm
\begin{eqnarray}
\vertiii{\bm{w}}^2=(\nabla_{w}\bm{w}, \nabla_{w}\bm{w}) + \kappa S(\bm{w},\bm{w})
\end{eqnarray}
for $\bm{w}= \begin{bmatrix} w \\ z \end{bmatrix} \in [W_0(\T_h)]^2$.
To show that $\vertiii{\cdot}$ is indeed a norm in $[W_0(\T_h)]^2$, it suffices to verify the length property for $\vertiii{\cdot}$.
To this end, assume $\vertiii{\bm{w}}=0$ for some $\bm{w}\in [W_0(\T_h)]^2$.
It follows that
\[
\nabla_w \bm{w}=\bm{0},\quad (\bm{w}|_{e_i} - \S(\bm{w})(M_i))|e_i|=\bm{0}
\]
for each edge $e_i\in \E_h$.
From $(\bm{w}|_{e_i} - \S(\bm{w})(M_i))|e_i|=\bm{0}$, we have
\[
\displaystyle \nabla \S(\bm{w})=\frac{1}{|T|}\sum_{i=1}^N \S(\bm{w})(M_i)|e_i| \otimes\bm{n}_i=\frac{1}{|T|}\sum_{i=1}^N \bm{w}|_{e_i}|e_i|\otimes\bm{n}_i=\nabla_w \bm{w}=\bm{0},
\]
where $\bm{x}\otimes\bm{y} = \{x_iy_j\}_{2\times 2}$ for any two vectors $\bm{x}=(x_1,x_2)$ and $\bm{y}=(y_1,y_2)$.
Therefore, $\S(\bm{w})$ has constant value on each element $T$, and so does $\bm{w}$ on the boundary of the element. Moreover, since $\bm{w}=\bm{0}$ on $\partial \Omega$, we then have $\bm{w}\equiv \bm{0}$ on $\Omega$ as $\Omega$ is assumed to be connected.

It is clear that $a(\bm{w},\bm{w}) = \vertiii{\bm{w}}^2$ for any
$\bm{w}\in [W_0(\T_h)]^2$. It follows from the definition of $\vertiii{\cdot}$ and the usual Cauchy-Schwarz inequality that the following boundedness and coercivity hold true for the bilinear form $a(\cdot,\cdot)$.
\begin{lemma}\label{lemma.blinear.a}
For any $\bm{u}, \bm{w}\in [W_0(\T_h)]^2$, we have
\begin{eqnarray}
&&|a(\bm{u},\bm{w})| \leq \vertiii{\bm{u}}\vertiii{\bm{w}},\label{equ.boundedness}\\
&& a(\bm{w},\bm{w}) \geqq \vertiii{\bm{w}}^2.
\end{eqnarray}
\end{lemma}

For the bilinear form $b(\cdot,\cdot)$, we have the following result on the {\em inf-sup} condition.
\begin{lemma}\label{lemma.blinear.b}
There exists a positive constant $\beta$ independent of $h$ such that
\begin{eqnarray}\label{equ.infsup.infsup}
\sup_{\bm{v}\in W_0(\T_h)}\frac{b(\bm{v},w)}{\vertiii{\bm{v}}} \geq \beta \|w\|,
\end{eqnarray}
for all $w \in P_0(\T_h)$.
\end{lemma}
\begin{proof}
For any given $w\in P_0(\T_h) \subset L_0^2(\Omega)$, it is well known \cite{Brenner2007,Brezzi1991,Girault1986,Gunzburger1989} that there exists a vector-valued function $\tilde{\bm{v}}\in[H_0^1(\Omega)]^2$ such that
\begin{equation}
\label{infsup.finiteelement}
(\nabla \cdot \tilde{\bm{v}},w) \geq C_1\| w\|^2,\quad \|\tilde{\bm{v}}\|_1 \leq C_2 \|w\|.
\end{equation}
where $C_i>0$ are two constants depending only on the domain $\Omega$.
Set $\bm{v} = Q_b \tilde{\bm{v}} \in [W_0(\T_h)]^2$ where $Q_b$ is the $L^2$ projection operator from $[H_0^1(\Omega)]^2$ to $[W_0(\T_h)]^2$.
We define also $Q_0$ and $\mathbb{Q}_h$ the $L^2$ projection operators onto $[P_1(\T_h)]^2$ and $[P_0(\T_h)]^{2\times 2}$ respectively.
Then, for all $q \in [P_0(\T_h)]^{2\times 2}$, we have
\begin{eqnarray}
(\nabla_w Q_b \tilde{\bm{v}},q)_T
&=&\langle Q_b \tilde{\bm{v}},q \cdot \bn \rangle_{\partial T} ~~~~~~~~~~~~~~~~~~~~~~~~~~~~~~~~~~~~~\label{equ.projection.nabla}\\
&=&\langle \tilde{\bm{v}},q \cdot \bn \rangle_{\partial T} \nonumber\\
&=& (\nabla \tilde{\bm{v}},q)_T \nonumber\\
&=& (\mathbb{Q}_h (\nabla\tilde{\bm{v}}),q)_T \nonumber,
\end{eqnarray}
which implies $\nabla_w Q_b \tilde{\bm{v}} = \mathbb{Q}_h (\nabla\tilde{\bm{v}})$.
Thus, the first part of the norm $\vertiii{\cdot}$ verifies:
\begin{eqnarray}\label{equ.norm.1}
  \sum_{T\in\T_h}\|\nabla_w \bm{v} \|_T^2
= \sum_{T\in\T_h} \|\nabla_w Q_b \tilde{\bm{v}} \|_T^2
= \sum_{T\in\T_h}\| \mathbb{Q}_h\nabla \tilde{\bm{v}} \|_T^2\leq C \|\tilde{\bm{v}}\|^2_1.
\end{eqnarray}

Next, we have
\begin{eqnarray*}
\|Q_b(\S(\bm{v})-Q_b \bm{v})\|^2_{\partial T}
&=&\langle Q_b(\S(Q_b \tilde{\bm{v}})) - Q_b \tilde{\bm{v}}, Q_b(\S(Q_b \tilde{\bm{v}})) - Q_b \tilde{\bm{v}}\rangle_{\partial T}\\
&=&\langle Q_b(\S(Q_b \tilde{\bm{v}})) - Q_b \tilde{\bm{v}}, Q_b(Q_0 \tilde{\bm{v}}) - Q_b \tilde{\bm{v}}\rangle_{\partial T} \nonumber\\
&\leq & \|Q_b(\S(Q_b \tilde{\bm{v}})) - Q_b \tilde{\bm{v}}\|_{\partial T} \|Q_b(Q_0 \tilde{\bm{v}}) - Q_b \tilde{\bm{v}}\|_{\partial T}, \nonumber
\end{eqnarray*}
which gives
\begin{eqnarray*}
\|Q_b(\S(\bm{v})-Q_b \bm{v})\|^2_{\partial T}
&\leq& \|Q_b(Q_0 \tilde{\bm{v}}) - Q_b \tilde{\bm{v}}\|^2_{\partial T} ~~~~~~~~~~~~~~~~~~~~~~~~~~~~~~~~\\
&\leq& \|Q_0 \tilde{\bm{v}} - \tilde{\bm{v}}\|^2_{\partial T}\nonumber\\
&\leq&  C_1 (h^{-1}\|Q_0 \tilde{\bm{v}} - \tilde{\bm{v}}\|^2_{T} + h\|\nabla (Q_0\tilde{\bm{v}}-\tilde{\bm{v}})\|^2_{T})\nonumber\\
&\leq&  C_2 h \|\tilde{\bm{v}}\|_1^2, \nonumber
\end{eqnarray*}
thus
\begin{eqnarray}\label{equ.norm.2}
\kappa {S}(\bm{v},\bm{v})=\kappa h^{-1}\sum_{T\in\T_h}\|Q_b(\S(\bm{v})-Q_b \bm{v})\|^2_{\partial T} \leq C_3 \|\tilde{\bm{v}}\|_1.~~~~~~~~~~~~~~~~~~~~
\end{eqnarray}
Combining the two equations \eqref{equ.norm.1} and \eqref{equ.norm.2} we arrive at
\begin{eqnarray}\label{equ.infsup.norm}
 \vertiii{\bm{v}} \leq C_0 \|\tilde{\bm{v}}\|_1,~~~~~~~~~~~~~~~~~~~~~~~~~~~~~~~
\end{eqnarray}
for some constant $C_0$.

On the other side, from \eqref{infsup.finiteelement} we have
\begin{eqnarray}\label{equ.infsup.bilinear}
   b(\bm{v},w)
  = (\nabla_w \cdot\bm{v},w)
& = & (\nabla_w \cdot Q_b\tilde{\bm{v}},w)\\
& = &  \langle Q_b\tilde{\bm{v}}\cdot \bm{n},w \rangle \nonumber\\
& = &  \langle \tilde{\bm{v}}\cdot \bm{n},w \rangle \nonumber\\
& = & (\nabla \cdot \tilde{\bm{v}},w ) \nonumber\\
& \geq & C_1\|w\|^2.\nonumber
\end{eqnarray}
Using the above equation, \eqref{equ.infsup.norm}, and \eqref{infsup.finiteelement}, we obtain
\begin{eqnarray}
\frac{ |b(\bm{v},w)|}{\vertiii{\bm{v}} } \geq \frac{C_1\|w\|^2}{C_0 \|\tilde{\bm{v}}\|_1}\geq \beta \|w\|,~~~~~~~~~~~~
\end{eqnarray}
for a positive constant $\beta$. This completes the proof of the lemma.
\end{proof}

\medskip
It follows from Lemma \ref{lemma.blinear.a} and Lemma \ref{lemma.blinear.b} that the following solvability holds true for the simplified weak Galerkin algorithm \eqref{equ.Stokes-FD-SWG}. Readers are referred to \cite{WangYe_2013,wy3655,wy-stokes} for a detailed discussion on the original weak Galerkin finite element method.
\begin{theorem}
The numerical scheme \eqref{equ.Stokes-FD-SWG} has one and only one solution for any positive stabilization parameter $\kappa>0$.
\end{theorem}

\medskip

Due to the connection between the finite difference scheme \eqref{equ.Stokes_FD-NEW} and the SWG finite element method, we have the following result for the solvability of the finite difference method.
\begin{theorem}\label{DF-uniq}
For any given stabilization parameter $\kappa>0$, the finite difference scheme \eqref{equ.Stokes_FD-NEW} has one and only one solution. The same conclusion holds true for the five-point finite difference scheme \eqref{equ.Stokes_FDS}.
\end{theorem}

\section{Error Estimates}\label{sectionEEStokes}
Let $\bm{u}$ and $p$ be the exact solution of the model problem \eqref{Equ.Stokes}, and denote by $\bm{u}_b=
\begin{bmatrix}
u_{b} \\
v_{b} \\
\end{bmatrix} \in [W_0(\T_h)]^2$ and $p_h \in P_0(\T_h)$ the numerical approximation arising from the SWG scheme \eqref{equ.Stokes-FD-SWG}.
Let $Q_b\bm{u}$ and $Q_h p$ be the $L^2$ projection of $\bm{u}$ and $p$ in the spaces $[W_0(\T_h)]^2$ and $P_0(\T_h)$, respectively. By the error function we mean the difference between the $L^2$ projection and the SWG approximations:
\begin{equation}\label{equ.error}
 \bm{e}=Q_b \bm{u}-\bm{u}_b,\; \eta=Q_h p-p_h.
\end{equation}

We now derive two equations for which the error functions $\bm{e}$ and $\eta$ must satisfy. The resulting equations are called the \textit{error equations}, which play an important role in the convergence analysis of the SWG scheme.

\begin{lemma}
Let $(\bm{u};p) \in [H^1(\Omega)]^2\times L^2(\Omega)$ be sufficiently smooth and satisfy the following equation
\begin{equation}\label{Equ.Stokes.1}
-\Delta \bm{u}+ \nabla p =\bm{f}  \text{ in } \Omega.
\end{equation}
The following equation holds true
\begin{eqnarray}\label{equ.lemma.wf}
& & \sum_{T}(\nabla_w Q_b \bm{u},\nabla_w \bm{v})_T - \sum_{T}(\nabla_w\cdot\bm{v}, Q_h p )_{T}\\
&= &(\bm{f}, \S(\bm{v})) + \sum_{T}\langle \frac{\partial \bm{u}}{\partial \bm{n}}- \mathbb{Q}_h(\nabla \bm{u})\cdot\bm{n},\S(\bm{v})-\bm{v}\rangle_{\partial T}\nonumber\\
&& - \sum_{T} \langle (\S(\bm{v})-\bm{v})\cdot \bm{n},p - Q_h p\rangle_{\partial T}, \nonumber
\end{eqnarray}
for all $\bm{v} \in [W_0(\T_h)]^2$, where $\mathbb{Q}_h$ is the $L^2$ projection operator onto $[P_0(\T_h)]^2$.
\end{lemma}

\begin{proof}
From equation \eqref{equ.projection.nabla}, we have
\begin{eqnarray}
&&\sum_{T}(\nabla_w Q_b \bm{u},\nabla_w \bm{v})_T\nonumber\\
&=& \sum_{T}(\mathbb{Q}_h(\nabla \bm{u}),\nabla_w \bm{v})_T \label{equ.wg.wf.1}\\
&=& \sum_{T} \langle  \mathbb{Q}_h(\nabla \bm{u})\cdot\bm{n},\bm{v}\rangle_{\partial T} \nonumber\\
&=& \sum_{T} \langle  \mathbb{Q}_h(\nabla \bm{u})\cdot\bm{n},\bm{v}\rangle_{\partial T}-\langle \mathbb{Q}_h(\nabla \bm{u})\cdot\bm{n},\S(\bm{v})\rangle_{\partial T} + \langle \mathbb{Q}_h(\nabla \bm{u})\cdot\bm{n},\S(\bm{v})\rangle_{\partial T} \nonumber\\
&=& \sum_{T} \langle \mathbb{Q}_h(\nabla \bm{u})\cdot\bm{n},\bm{v}-\S(\bm{v})\rangle_{\partial T} +(\nabla \bm{u},\nabla \S(\bm{v}))_{T} \nonumber\\
&=& \sum_{T} \langle \mathbb{Q}_h(\nabla \bm{u})\cdot\bm{n},\bm{v}-\S(\bm{v})\rangle_{\partial T} +(-\nabla \cdot(\nabla \bm{u}) ,\S(\bm{v}))_{T} + \langle \frac{\partial \bm{u}}{\partial \bm{n}},\S(\bm{v})\rangle_{\partial T}\nonumber\\
&=& (-\Delta\bm{u} ,\S( \bm{v}))+ \sum_{T}\langle \frac{\partial \bm{u}}{\partial \bm{n}}- \mathbb{Q}_h(\nabla \bm{u})\cdot\bm{n},\S(\bm{v})-\bm{v}\rangle_{\partial T}, \nonumber
\end{eqnarray}
for all $\bm{v}\in [W_0(\T_h)]^2$, where $\bm{v}|_{\partial\Omega}=0$ has been used. For $(\nabla_w\cdot\bm{v}, Q_h p )_{T}$, we have
\begin{equation}\label{equ.wg.wf.2}
\begin{split}
&\sum_{T}(\nabla_w\cdot\bm{v}, Q_h p )_{T}\\
=& \sum_{T} -(\S(\bm{v}),\nabla Q_h p)_{T} + \langle\bm{v}\cdot \bm{n}, Q_h p \rangle_{\partial T} \\
=& \sum_{T}(\nabla \S(\bm{v}), Q_h p)_{T} + \langle (\bm{v} -\S(\bm{v}))\cdot \bm{n}, Q_h p \rangle_{\partial T} \\
=& \sum_{T}(\nabla \S(\bm{v}),  p)_{T}    + \langle (\bm{v} -\S(\bm{v}))\cdot \bm{n}, Q_h p \rangle_{\partial T}  \\
=& \sum_{T} -( \S(\bm{v}), \nabla p))_{T}  + \langle \S(\bm{v})\cdot \bm{n},p \rangle_{\partial T} + \langle(\bm{v} -\S(\bm{v}))\cdot \bm{n}, Q_h p\rangle_{\partial T} \\
=& \sum_{T}-(\nabla p, \S(\bm{v}))_{T} + \langle (\S(\bm{v})-\bm{v})\cdot \bm{n},p - Q_h p\rangle_{\partial T}\\
=& -(\nabla p, \S(\bm{v})) + \sum_{T} \langle (\S(\bm{v})-\bm{v})\cdot \bm{n},p - Q_h p\rangle_{\partial T}
\end{split}
\end{equation}
for all $\bm{v}\in [W_0(\T_h)]^2$, where we have used the definition for the weak divergence operator ($\nabla_w\cdot$) and the fact that $\nabla Q_h p=0$ in the second line, the identity of $\sum_{T} \langle \bm{v}\cdot \bm{n}, p \rangle_{\partial T} =0$ in the sixth line.

We now test \eqref{Equ.Stokes.1} with $\S(\bm{v})$, $\bm{v} \in [W_0(\T_h)]^2$, to obtain
\begin{eqnarray}\label{equ.wg.wf}
(-\Delta \bm{u}, \S(\bm{v}) ) + (\nabla p, \S(\bm{v})) =(\bm{f}, \S(\bm{v})).
\end{eqnarray}
Substituting equations \eqref{equ.wg.wf.1} and \eqref{equ.wg.wf.2} into \eqref{equ.wg.wf} yields
\begin{eqnarray*}
& & \sum_{T}(\nabla_w Q_b \bm{u},\nabla_w \bm{v})_T - \sum_{T}(\nabla_w\cdot\bm{v}, Q_h p )_{T}\\
&= &(\bm{f}, \S(\bm{v})) + \sum_{T}\langle \frac{\partial \bm{u}}{\partial \bm{n}}- \mathbb{Q}_h(\nabla \bm{u})\cdot\bm{n},\S(\bm{v})-\bm{v}\rangle_{\partial T}
- \sum_{T} \langle (\S(\bm{v})-\bm{v})\cdot \bm{n},p - Q_h p\rangle_{\partial T}, \nonumber
\end{eqnarray*}
which completes the proof of the lemma.
\end{proof}

The following is a result on the error equation for the SWG scheme \eqref{equ.Stokes-FD-SWG}.
\begin{lemma}
Let $\bm{e}$ and $\eta$ be the error functions for the numerical solution arising from the SWG scheme \eqref{equ.Stokes-FD-SWG}, as given in
\eqref{equ.error}. Then, we have
\begin{eqnarray}
a(\bm{e},\bm{v})-b(\bm{v},\eta) &=&\varphi_{\bm{u},p}(\bm{v}), \label{equ.error.1}\\
b(\bm{e},w)&=&0,\label{equ.error.2}
\end{eqnarray}
for all $\bm{v}\in [W_0(\T_h)]^2$ and $w \in P_0(\T_h)$, where
\begin{equation}\label{equ.error.3}
\begin{split}
\varphi_{\bm{u},p}(\bm{v}) =& \ \kappa {S}(Q_b\bm{u},\bm{v})
+ \sum_{T}\langle \frac{\partial \bm{u}}{\partial \bm{n}}- \mathbb{Q}_h(\nabla \bm{u})\cdot\bm{n},\S(\bm{v})-\bm{v}\rangle_{\partial T}\\
&\ - \sum_{T} \langle (\S(\bm{v})-\bm{v})\cdot \bm{n},p - Q_h p\rangle_{\partial T}.
\end{split}
\end{equation}
\end{lemma}
\begin{proof}
By adding $\kappa {S}(Q_b\bm{u},\bm{v})$ to both side of \eqref{equ.lemma.wf} we obtain
\begin{equation} \label{equ.lemma.wf.2}
\begin{split}
 & \ a(Q_b(\bm{u}),\bm{v}) - b(\bm{v},Q_h p)\\
 = & \ \kappa {S}(Q_b\bm{u},\bm{v})+\sum_{T}(\nabla_w Q_b \bm{u},\nabla_w \bm{v})_T - \sum_{T}(\nabla_w\cdot\bm{v}, Q_h p )_{T}\\
=& \ (\bm{f}, \S(\bm{v})) + \kappa {S}(Q_b\bm{u},\bm{v}) + \sum_{T}\langle \frac{\partial \bm{u}}{\partial \bm{n}}- \mathbb{Q}_h(\nabla \bm{u})\cdot\bm{n},\S(\bm{v})-\bm{v}\rangle_{\partial T}\\
&\ \ - \sum_{T} \langle (\S(\bm{v})-\bm{v})\cdot \bm{n},p - Q_h p\rangle_{\partial T}.
\end{split}
\end{equation}
Now subtracting the first equation of \eqref{equ.Stokes-FD-SWG} from  \eqref{equ.lemma.wf.2} yields
\begin{eqnarray*}
& & a(\bm{e},\bm{v})-b(\bm{v},\eta) =\varphi_{\bm{u},p}(\bm{v}) \qquad \forall \bm{v}\in [W_0(\T_h)]^2.
\end{eqnarray*}
This completes the derivation of the error equation \eqref{equ.error.1}.

As to the second one \eqref{equ.error.2}, we have
\begin{eqnarray*}
b(\bm{e},w)=(\nabla_w\cdot\bm{e},w)
&=&(\nabla_w\cdot(Q_b\bm{u}-\bm{u}_b),w)\\
&=&(\nabla_w\cdot Q_b\bm{u},w)-(\nabla_w\cdot\bm{u}_b,w) \nonumber\\
&=&(Q_h\nabla \cdot\bm{u},w)-0 \nonumber\\
&=&(Q_h (0),w)-0 \nonumber\\
&=&0,\nonumber
\end{eqnarray*}
for all $w \in P_0(\T_h)$. This completes the proof of the lemma.
\end{proof}

\subsection{Error estimates in $H^1$} The goal here is to establish an optimal-order error estimate for the numerical solution $(\bm{u}_b; p_h)$ in a discrete $H^1$-norm for the velocity and the $L^2$ norm for the pressure. The result can be stated as follows.

\begin{theorem}\label{thm.error.1}
Let $(\bm{u};p) \in [H_0^1(\Omega)\cap H^{2}(\Omega)]^2\times (L_0^2(\Omega)\cap H^{1}(\Omega))$ be the solution of \eqref{Equ.Stokes}, and $(\bm{u}_b;p_h) \in [W_0(\T_h)]^2 \times P_0(\T_h)$ be the solution of the SWG scheme \eqref{equ.Stokes-FD-SWG}, respectively.
Then, the following error estimate holds true
\begin{eqnarray} \label{equ.thm.error.1}
\vertiii{Q_b\bm{u} -\bm{u}_b} +\|Q_h p - p_h\| \leq Ch(\|\bm{u}\|_2+\|p\|_1).
\end{eqnarray}
\end{theorem}
\begin{proof}
By letting $\bm{v}=\bm{e}$ in \eqref{equ.error.1} and then using \eqref{equ.error.2} with $w=\eta$, we have
\begin{equation}\label{equ.EE.1}
\begin{split}
\vertiii{\bm{e}}^2=&\ \varphi_{\bm{u},p}(\bm{e})\\
=&\ \kappa {S}(Q_b\bm{u},\bm{e}) + \sum_{T}\langle \frac{\partial \bm{u}}{\partial \bm{n}}- \mathbb{Q}_h(\nabla \bm{u})\cdot\bm{n},\S(\bm{e})-\bm{e}\rangle_{\partial T}\\
&\ - \sum_{T} \langle (\S(\bm{e})-\bm{e})\cdot \bm{n},p - Q_h p\rangle_{\partial T}.
\end{split}
\end{equation}
From the Cauchy-Schwarz inequality and \eqref{equ.norm.2}, for the first term in \eqref{equ.EE.1}, we have
\begin{equation} \label{eq.EE.H1.1}
\begin{split}
|S(Q_b\bm{u},\bm{e})|=&\ \bigg|h^{-1}\sum_T \langle Q_b\S(Q_b\bm{u})- Q_b\bm{u}, Q_b\S(\bm{e})-Q_b\bm{e}\rangle_{\partial T}\bigg|~~\\
=&\ \bigg|h^{-1}\sum_T \langle Q_b(Q_0\bm{u})- Q_b\bm{u}, Q_b\S(\bm{e})-Q_b\bm{e}\rangle_{\partial T}\bigg|\\
=&\ \bigg|h^{-1}\sum_T \langle Q_b(Q_0\bm{u}- \bm{u}), Q_b\S(\bm{e})-Q_b\bm{e}\rangle_{\partial T}\bigg|\\
\leq &\ \left(h^{-2}\|Q_0\bm{u}- \bm{u}\| +\|\nabla(Q_0\bm{u}- \bm{u})\| \right)^{\frac{1}{2}} S(\bm{e},\bm{e})^{\frac{1}{2}}\\
\leq &\ C h\|\bm{u}\|_2 \vertiii{\bm{e}}.
\end{split}
\end{equation}

Next, for any constant tensor $\bm{\phi}$, we have from the definition of the weak gradient that
\begin{eqnarray*}
(\nabla_w\cdot \bm{e}, \bm{\phi})_T &=& \langle \bm{e}, \bm{\phi}\cdot\bm{n}\rangle_\pT\\
&=& \langle\bm{e}- \S(\bm{e}), \bm{\phi}\cdot\bm{n}\rangle_\pT + \langle \S(\bm{e}), \bm{\phi}\cdot\bm{n}\rangle_\pT\\
&=& \langle\bm{e}- Q_b\S(\bm{e}), \bm{\phi}\cdot\bm{n}\rangle_\pT + (\nabla \S(\bm{e}), \bm{\phi})_T,
\end{eqnarray*}
which gives
\begin{eqnarray*}
\|\nabla \S(\bm{e})\|_T^2 \leq  \|\nabla_w \bm{e}\|_T^2 + C h^{-1}\|\bm{e}-Q_b\S(\bm{e})\|_\pT^2.
\end{eqnarray*}
It follows that
\begin{equation}\label{EQ:Hi-Feb27-001}
\begin{split}
\|\bm{e}-\S(\bm{e})\|_\pT & \leq \|\bm{e}-Q_b\S(\bm{e})\|_\pT + \|(I-Q_b)\S(\bm{e})\|_\pT \\
& \leq \|\bm{e}-Q_b\S(\bm{e})\|_\pT + C h^{1/2}\|\nabla \S(\bm{e})\|_T \\
& \leq C \|\bm{e}-Q_b\S(\bm{e})\|_\pT + C h^{1/2}\|\nabla_w \bm{e}\|_T.
\end{split}
\end{equation}
From the Cauchy-Schwarz inequality and the estimate \eqref{EQ:Hi-Feb27-001} we obtain
\begin{eqnarray*}
&&\bigg|\langle \frac{\partial \bm{u}}{\partial \bm{n}}- Q_0(\nabla \bm{u})\cdot\bm{n},\bm{e}-\S(\bm{e})\rangle_{\partial T}\bigg|\\
&\leq &  \|\frac{\partial \bm{u}}{\partial \bm{n}}- Q_0(\nabla \bm{u})\cdot\bm{n}\|_\pT \|\bm{e}-\S(\bm{e})\|_\pT\\
&\leq & C \|\frac{\partial \bm{u}}{\partial \bm{n}}- Q_0(\nabla \bm{u})\cdot\bm{n}\|_\pT \left( \|\bm{e}-Q_b\S(\bm{e})\|_\pT + h^{1/2} \|\nabla_w\bm{e}\|_T\right).
\end{eqnarray*}
Now summing over all the elements yields an estimate for the second term in \eqref{equ.EE.1}:
\begin{equation}\label{eq.EE.H1.2}
\begin{split}
&\sum_T\bigg|\langle \frac{\partial \bm{u}}{\partial \bm{n}}- Q_0(\nabla \bm{u})\cdot\bm{n},\bm{e}-\S(\bm{e})\rangle_{\partial_T}\bigg|\\
\leq & \sum_T \|\frac{\partial \bm{u}}{\partial \bm{n}}- Q_0(\nabla \bm{u})\cdot\bm{n}\|_\pT \left( \|\bm{e}-Q_b\S(\bm{e})\|_\pT + Ch^{1/2} \|\nabla_w\bm{e}\|_T\right)\\
\leq & C \left(\|\nabla \bm{u}-Q_0(\nabla \bm{u})\|^2_0 + h^2\|\nabla^2 \bm{u} \|_{0}\right)^{\frac{1}{2}}\left( \|\nabla_w \bm{e}\|^2+ \kappa S(\bm{e}, \bm{e})\right)^{\frac{1}{2}}\\
\leq & Ch\|\bm{u}\|_2\vertiii{\bm{e}}.
\end{split}
\end{equation}
Similarly, we have
\begin{eqnarray}\label{eq.EE.H1.3}
\left|\sum_{T} \langle (\S(\bm{e})-\bm{e})\cdot \bm{n},p - Q_h p\rangle_{\partial T}\right|&\leq & Ch\|p\|_1\vertiii{\bm{e}}.
\end{eqnarray}
Combining the estimates \eqref{eq.EE.H1.1}, \eqref{eq.EE.H1.2}, and \eqref{eq.EE.H1.3} with \eqref{equ.EE.1} yields the first part of the error estimate \eqref{equ.thm.error.1}:
\begin{eqnarray} \label{equ.thm.error.12}
\vertiii{Q_b\bm{u} -\bm{u}_b} \leq Ch(\|\bm{u}\|_2+\|p\|_1).
\end{eqnarray}

To estimate $\|\eta\|$, we use equation \eqref{equ.error.1} to obtain
\begin{eqnarray*}
b(\bm{v},\eta) =a(\bm{e},\bm{v})-\varphi_{\bm{u},p}(\bm{v}).
\end{eqnarray*}
Using the equation above, \eqref{eq.EE.H1.1}-\eqref{eq.EE.H1.3}, and \eqref{equ.boundedness}, we arrive at
\begin{eqnarray*}
|b(\bm{v},\eta)| \leq Ch(\|\bm{u}\|_2+\|p\|_1)\vertiii{\bm{v}}.
\end{eqnarray*}
Combining the above estimate with the {\em inf-sup} condition \eqref{equ.infsup.infsup} gives
\begin{eqnarray*}
\| \eta \| \leq Ch(\|\bm{u}\|_2+\|p\|_1),
\end{eqnarray*}
which, together with \eqref{equ.thm.error.12}, yields the desired estimate \eqref{equ.thm.error.1}.
\end{proof}

\subsection{Error estimates in $L^2$}
To derive an $L^2$-error estimate for the velocity approximation, we consider the problem of seeking $(\bm{\psi};\xi)$ such that
\begin{equation}\label{Equ.Stokes.dual}
\left\{
\begin{array}{rrl}
-\Delta \bm{\psi}+ \nabla \xi &=\bm{g}  &\quad \text{ in }\ \Omega\\
{\rm div} \bm{\psi}&=0                  &\quad \text{ in }\ \Omega\\
 \bm{\psi}&=0                           &\quad \text{ on }\ \Gamma=\partial \Omega.
\end{array}
\right.
\end{equation}
Assume that the problem \eqref{Equ.Stokes.dual} has the $[H^2(\Omega)]^2\times H^1(\Omega)$-regularity in the sense that the solution
$(\bm{\psi};\xi) \in [H^2(\Omega)]^2\times H^1(\Omega)$ and the following a priori estimate holds true:
\begin{equation}\label{equ.regularity}
\|\bm{\psi}\|_2+\|\xi\|_1\leq C\|\bm{g}\|
\end{equation}
for any $\bm{g}\in [L^2(\Omega)]^2$.

\begin{theorem}\label{thm.error.3}
Let $(\bm{u};p) \in [H_0^1(\Omega)\cap H^{2}(\Omega)]^2\times (L_0^2(\Omega)\cap H^{1}(\Omega))$ be the solution of \eqref{Equ.Stokes}, and $(\bm{u}_b;p_h) \in [W_0(\T_h)]^2 \times P_0(\T_h)$ be the solution of the SWG scheme \eqref{equ.Stokes-FD-SWG}, respectively.
Then, the following error estimate holds true:
\begin{eqnarray} \label{equ.thm.error.3}
\|\S(Q_b \bm{u}) - \S(\bm{u}_b)\| \leq Ch^2 (\|\bm{u}\|_2+\|p\|_1).
\end{eqnarray}
\end{theorem}

\begin{proof}
Let $(\bm{\psi};\xi)$ be the solution of \eqref{Equ.Stokes.dual} with $\bm{g}=\S(\bm{e})=\S(Q_h\bm{u}-\bm{u}_b)$.
From \eqref{equ.lemma.wf}, we have for any $\bm{v} \in [W_0(\T_h)]^2$
\begin{equation}\label{equ.lemma.wf.EE.1}
\begin{split}
& \sum_{T}(\nabla_w Q_b \bm{\psi},\nabla_w \bm{v})_T - \sum_{T}(\nabla_w\cdot\bm{v}, Q_h \xi )_{T}\\
= &(\S(\bm{e}), \S(\bm{v})) + \sum_{T}\langle \frac{\partial \bm{\psi}}{\partial \bm{n}}- \mathbb{Q}_h(\nabla \bm{\psi})\cdot\bm{n},\S(\bm{v})-\bm{v}\rangle_{\partial T}\\
& - \sum_{T} \langle (\S(\bm{v})-\bm{v})\cdot \bm{n},\xi - Q_h \xi\rangle_{\partial T}.
\end{split}
\end{equation}
In particular, by letting $\bm{v} =\bm{e}$ we obtain
\begin{equation*}
\begin{split}
& \sum_{T}(\nabla_w Q_b \bm{\psi},\nabla_w \bm{e})_T - \sum_{T}(\nabla_w\cdot\bm{e}, Q_h \xi )_{T}\\
= &\ (\S(\bm{e}), \S(\bm{e})) + \sum_{T}\langle \frac{\partial \bm{\psi}}{\partial \bm{n}}- \mathbb{Q}_h(\nabla \bm{\psi})\cdot\bm{n},\S(\bm{e})-\bm{e}\rangle_{\partial T}\\
& \; \; - \sum_{T} \langle (\S(\bm{e})-\bm{e})\cdot \bm{n},\xi - Q_h \xi\rangle_{\partial T}.
\end{split}
\end{equation*}
By adding and subtracting $\kappa {S}(Q_b\bm{\psi},\bm{e})$ in the above equation we arrive at
\begin{equation}\label{EQ:July01:002}
\begin{split}
 \|\S(\bm{e})\|^2 =&\ a(Q_b \bm{\psi},\bm{e}) -b(\bm{e},Q_h \xi) \\
&- \kappa {S}(Q_b\bm{\psi},\bm{e})- \sum_{T}\langle \frac{\partial \bm{\psi}}{\partial \bm{n}}- \mathbb{Q}_h(\nabla \bm{\psi})\cdot\bm{n},\S(\bm{e})-\bm{e}\rangle_{\partial T}\\
& + \sum_{T} \langle (\S(\bm{e})-\bm{e})\cdot \bm{n},\xi - Q_h \xi\rangle_{\partial T} \\
=&\ a(Q_b \bm{\psi},\bm{e}) -b(\bm{e},Q_h \xi) -\varphi_{\bm{\psi},\xi}(\bm{e}).
\end{split}
\end{equation}
From \eqref{equ.error.2} and the second equation of \eqref{Equ.Stokes.dual} we have
\begin{eqnarray}
b(\bm{e},Q_h \xi)=0,\qquad b( Q_b \bm{\psi},\eta)=0.~~
\end{eqnarray}
Combining the above two equations gives
\begin{eqnarray}
 \|\S(\bm{e})\|^2 &=& a(\bm{e},Q_b \bm{\psi})-b( Q_b \bm{\psi},\eta)-\varphi_{\bm{\psi},\xi}(\bm{e}).
\end{eqnarray}
Using \eqref{equ.error.1} and the equation above, we obtain
\begin{eqnarray}\label{equ.stokes.L2.error}
 \|\S(\bm{e})\|^2 &=& \varphi_{\bm{u},p}(Q_b \bm{\psi}) -\varphi_{\bm{\psi},\xi}(\bm{e}).
\end{eqnarray}
The rest of the proof will establish some estimates for the right-hand side of \eqref{equ.stokes.L2.error}.

The following are two useful estimates in the forthcoming analysis. For any $\bm{v}\in [W(T)]^2$, we have
\begin{eqnarray}
  \|\bm{v} - Q_b\S(\bm{v})\|_{\pT} & \leq & \|\bm{v} - Q_b\bm{\phi}\|_{\pT}\quad \forall \bm{\phi}\in [P_1(T)]^2,\label{EQ:July01:09}\\
  \|\S(\bm{v})\|_\pT & \leq & C \|\bm{v}\|_\pT. \label{EQ:July01:10}
\end{eqnarray}
The inequality \eqref{EQ:July01:09} is a direct consequence of the definition of $\S(\bm{v})$ as the projection of $\bm{v}$ into the space of linear functions with respect to the norm $\|\cdot\|_\pT$. The inequality \eqref{EQ:July01:10} can be derived by turning $\|\S(\bm{v})\|_\pT$ into a discrete norm using the midpoints $M_i$.

Each of the terms in $\varphi_{\bm{u},p}(Q_b \bm{\psi})$; namely,
\begin{eqnarray*}
& & \varphi_{\bm{u},p}(Q_b \bm{\psi})\\
&=&\kappa {S}(\bm{u}, Q_b \bm{\psi})
+\sum_{T}\langle \frac{\partial \bm{u}}{\partial \bm{n}}- Q_h(\nabla \bm{u})\cdot\bm{n},\S(Q_b\bm{\psi})-Q_b\bm{\psi}\rangle_{\partial T}\\
&& -\sum_{T} \langle (\S(Q_b \bm{\psi})-Q_b \bm{\psi})\cdot \bm{n},p - Q_h p\rangle_{\partial T},\nonumber
\end{eqnarray*}
can be handled as follows:
\begin{itemize}
\item For the stability term $\kappa {S}(\bm{u},Q_b\bm{\psi})$, one has from the orthogonality property of the extension operator $\S(\cdot)$
\begin{equation*}%\label{equ.EL2.1}
\begin{split}
|{S}(\bm{u}, Q_b \bm{\psi})|=&\ \bigg|h^{-1}\sum_T\langle Q_b \S(\bm{u})- Q_b\bm{u}, Q_b\S(Q_b \bm{\psi}) - Q_b \bm{\psi}\rangle_\pT \bigg| \\
=&\ \bigg|h^{-1}\sum_T\langle Q_b(Q_0 \bm{u})- Q_b\bm{u}, Q_b\S(Q_b \bm{\psi}) - Q_b \bm{\psi}\rangle_\pT \bigg| \\
\leq &\ h^{-1}\sum_T \|Q_0\bm{u}-\bm{u}\|_\pT \|Q_b\S(Q_b \bm{\psi}) - Q_b \bm{\psi}\|_\pT.
\end{split}
\end{equation*}
From \eqref{EQ:July01:09} with $\bm{v}=Q_b \bm{\psi}$ and $\bm{\phi}=Q_0\bm{\psi}$ we have
\begin{equation}\label{equ.EL2.1}
\begin{split}
|{S}(\bm{u}, Q_b \bm{\psi})| \leq &\ h^{-1}\sum_T \|Q_0\bm{u}-\bm{u}\|_\pT \|Q_b\S(Q_b \bm{\psi}) - Q_b \bm{\psi}\|_\pT\\
\leq &\ h^{-1}\sum_T \|Q_0\bm{u}-\bm{u}\|_\pT \|Q_b(Q_0\bm{\psi}) - Q_b \bm{\psi}\|_\pT\\
\leq& \ h^{-1} \left( \sum_{T}\|Q_0 \bm{u}- \bm{u}\|_{\partial T}^2 \right)^{\frac{1}{2}}\left( \sum_{T}\|Q_0 \bm{\psi}- \bm{\psi}\|_{\partial T}^2 \right)^{\frac{1}{2}} \\
\leq&\ Ch^2\|\bm{u}\|_2\|\bm{\psi}\|_2.
\end{split}
\end{equation}
\item For the second term in $\varphi_{\bm{u},p}(Q_b \bm{\psi})$, since $\bm{\psi}=0$ on $\partial \Omega$, we have
\begin{eqnarray*}
 \sum_{T}\langle \frac{\partial \bm{u}}{\partial \bm{n}}- \mathbb{Q}_h(\nabla \bm{u})\cdot\bm{n}, Q_b\bm{\psi}-\bm{\psi}\rangle_{\partial T}
=\sum_{T}\langle \frac{\partial \bm{u}}{\partial \bm{n}}, Q_b\bm{\psi}-\bm{\psi}\rangle_{\partial T}
=0,
\end{eqnarray*}
which leads to
\begin{eqnarray*}
&& \bigg|\sum_{T}\langle \frac{\partial \bm{u}}{\partial \bm{n}}- \mathbb{Q}_h(\nabla \bm{u})\cdot\bm{n},\S(Q_b\bm{\psi})-Q_b\bm{\psi}\rangle_{\partial T}\bigg|\\
&=&\bigg|\sum_{T}\langle \frac{\partial \bm{u}}{\partial \bm{n}}- \mathbb{Q}_h(\nabla \bm{u})\cdot\bm{n},\S(Q_b\bm{\psi})-\bm{\psi}\rangle_{\partial T}\bigg| \nonumber\\
&=&\bigg|\sum_{T}\langle \frac{\partial \bm{u}}{\partial \bm{n}}- \mathbb{Q}_h(\nabla \bm{u})\cdot\bm{n},\S(Q_b\bm{\psi})-\S(Q_b(Q_0\bm{\psi})) + Q_0\bm{\psi} - \bm{\psi}\rangle_{\partial T}\bigg|,
\end{eqnarray*}
where we have used the fact that
$$
\S(Q_b(Q_0\bm{\psi})) = Q_0\bm{\psi}.
$$
It follows from \eqref{EQ:July01:10} that
\begin{equation}\label{equ.EL2.2}
\begin{split}
& \bigg|\sum_{T}\langle \frac{\partial \bm{u}}{\partial \bm{n}}- \mathbb{Q}_h(\nabla \bm{u})\cdot\bm{n},\S(Q_b\bm{\psi})-Q_b\bm{\psi}\rangle_{\partial T}\bigg|\\
 \leq & \left( \sum_{T} \|\nabla \bm{u} \cdot \bm{n} - \mathbb{Q}_h(\nabla \bm{u})\cdot\bm{n} \|_{\partial T}^2\right)^{\frac{1}{2}} \\
& \cdot \left( \sum_{T} \| Q_0\bm{\psi} - \bm{\psi}\|_{\partial T}^2 + \|\S(Q_b(\bm{\psi} - Q_0\bm{\psi} ))\|_{\partial T}^2\right)^{\frac{1}{2}}, \\
 \leq & C \left( \sum_{T} \|\nabla \bm{u} \cdot \bm{n} - \mathbb{Q}_h(\nabla \bm{u})\cdot\bm{n} \|_{\partial T}^2\right)^{\frac{1}{2}}
         \left(\sum_{T} \| Q_0\bm{\psi} - \bm{\psi}\|_{\partial T}^2 \right)^{\frac{1}{2}} \\
 \leq & Ch^2\|u\|_2 \|\bm{\psi}\|_2.
 \end{split}
\end{equation}
\item As to the last term in $\varphi_{\bm{u},p}(Q_b \bm{\psi})$, we have
\begin{equation}\label{equ.EL2.3}
\begin{split}
& \bigg|\sum_{T} \langle (\S(Q_b \bm{\psi})-Q_b \bm{\psi})\cdot \bm{n},p - Q_h p\rangle_{\partial T}\bigg| \\
= & \bigg|\sum_{T} \langle (\S(Q_b \bm{\psi})- \bm{\psi})\cdot \bm{n},p - Q_h p\rangle_{\partial T}\bigg| \\
=&\bigg|\sum_{T} \langle (\S(Q_b \bm{\psi}) - \S(Q_bQ_0\bm{\psi}))\cdot\bm{n} + (Q_0 \bm{\psi}- \bm{\psi})\cdot \bm{n},p - Q_h p\rangle_{\partial T} \bigg| \\
\leq & C \left(\sum_{T} h\|p-Q_hp\|^2_{\partial T} \right)^{\frac{1}{2}}\left(\sum_{T} h^{-1}\|Q_0 \bm{\psi}- \bm{\psi}\|^2_{\partial T} \right)^{\frac{1}{2}}\\
\leq & Ch^2\|p\|_1 \|\bm{\psi}\|_2.
\end{split}
\end{equation}
\end{itemize}
The estimates \eqref{equ.EL2.1}, \eqref{equ.EL2.2} and \eqref{equ.EL2.3}, together with the regularity \eqref{equ.regularity}, collectively yield
\begin{eqnarray}\label{equ.stokes.L2.1}
|\varphi_{\bm{u},p}(Q_b \bm{\psi})|
&\leq& Ch^{2}(\|u\|_2+\|p\|_1)\|\bm{\psi}\|_2\\
&\leq& Ch^{2}(\|u\|_2+\|p\|_1)\|\S(\bm{e})\|.\nonumber
\end{eqnarray}

Next, using the same argument as in the proof of Theorem \ref{thm.error.1} and the $H^2$-regularity assumption \eqref{equ.regularity} we arrive at
\begin{eqnarray*}%\label{equ.stokes.L2.2}
|\varphi_{\bm{\psi},\xi}(\bm{e})|\leq Ch(\|\psi\|_2 +\|\xi\|_1)\vertiii{\bm{e}}\leq Ch \|\S(\bm{e})\|\vertiii{\bm{e}}.
\end{eqnarray*}
It then follows from the estimate \eqref{equ.thm.error.1} that
\begin{eqnarray}\label{equ.stokes.L2.2}
|\varphi_{\bm{\psi},\xi}(\bm{e})|\leq Ch^2(\|\bm{u}\|_2 +\|p\|_1)\|\S(\bm{e})\|.
\end{eqnarray}

Finally, substituting \eqref{equ.stokes.L2.1} and \eqref{equ.stokes.L2.2} into \eqref{equ.stokes.L2.error} gives
\begin{eqnarray}
\|\S(\bm{e})\|^2\leq Ch^2(\|u\|_2+\|p\|_1)\|\S(\bm{e})\|,
\end{eqnarray}
which gives rise to the desired error estimate \eqref{equ.thm.error.3}.
\end{proof}

\section{Superconvergence}\label{sectionSuperC}
Next we shall derive a superconvergence for the approximate velocity and pressure when the finite element partition consists of only rectangular elements. The result is based on the framework developed in \cite{LiDanWW} for the second order elliptic equation, with a particular attention paid to the pressure term.

\begin{theorem}\label{thm.error.2}
Let $(\bm{u};p) \in [H_0^1(\Omega)\cap H^{3}(\Omega)]^2\times (L_0^2(\Omega)\cap H^{2}(\Omega))$ be the solution of \eqref{Equ.Stokes}, and $(\bm{u}_b;p_h) \in [W_0(\T_h)]^2 \times P_0(\T_h)$ be the solution of the SWG scheme \eqref{equ.Stokes-FD-SWG} on rectangular meshes, respectively.
Then, the following superconvergence holds true:
\begin{eqnarray} \label{equ.thm.error.2}
\|\nabla_w(Q_h \bm{u} - \bm{u}_b )\|_0 + \|Q_h p - p_h\|_0 \leq Ch^{3/2} (\|\bm{u}\|_3 + \|p\|_2).
\end{eqnarray}
\end{theorem}
\begin{proof}
The proof relies on a refined treatment of the linear functional $\varphi_{\bm{u},p}(\cdot)$ in the error equation (\ref{equ.error.1}). To this end, consider the term $\sum_{T} \langle (\S(\bm{v})-\bm{v})\cdot \bm{n},p - Q_h p\rangle_{\partial T}$ on the right-hand side of  \eqref{equ.error.3}, where $\bm{v} \in [W_0(\T_h)]^2$ is an arbitrary test function. Recall that, from \eqref{EQ:S-Property:002} and \eqref{EQ:S-Property:008}, $\S(\bm{v})-\bm{v}$ has the same value at the center of $e_1$ and $e_2$ (respectively, $e_3$ and $e_4$). As shown in Fig. \ref{fig:rectangular}, we have $\bm{n}_{1}=- \bm{n}_{2}$ (respectively, $\bm{n}_{3}=- \bm{n}_{4}$ for the two horizontal edges). As $Q_h p$ is constant-valued on $T$, we thus have
\begin{eqnarray*}
\langle(\S(\bm{v})-\bm{v})\cdot\bm{n}, Q_h p\rangle_{\partial T} =0.
\end{eqnarray*}
Furthermore, with $\bm{v}=(v_1,v_2)^t$, we obtain
\begin{equation}\label{EQ:div:001}
\begin{split}
& \langle(\S(\bm{v})-\bm{v})\cdot \bm{n},p- Q_h p\rangle_{\partial T} \\
=\ &\langle(\S(\bm{v})-\bm{v})\cdot \bm{n},p\rangle_{\partial T}\\
=\ &\langle\S(v_1)-v_1,p\rangle_{e_2} - \langle\S(v_1)-v_1,p\rangle_{e_1}\\
\ &+ \langle\S(v_2)-v_2,p\rangle_{e_4} - \langle\S(v_2)-v_2,p\rangle_{e_3}\\
=\ & (\chi_1, \partial_x p)_T + (\chi_2, \partial_y p)_T,
\end{split}
\end{equation}
where $\chi_1$ is the constant extension of $\S(v_1)-v_1$ along the $x$-direction; i.e.,
$$
\chi_1(x,y): = (\S(v_1)-v_1)|_{e_1}(y) = (\S(v_1)-v_1)|_{e_2}(y),\qquad (x,y)\in T
$$
and $\chi_2$ is defined analogously as the constant extension of $\S(v_2)-v_2$ along the $y$-direction:
$$
\chi_2(x,y): = (\S(v_2)-v_2)|_{e_3}(x) = (\S(v_2)-v_2)|_{e_4}(x),\qquad (x,y)\in T.
$$
The equations \eqref{EQ:July01:001} show that the above extensions are well-defined.

The two terms on the right-hand side of \eqref{EQ:div:001} can be handled as follows. As $\chi_1$ is linear in $y$ and constant in $x$, we have
\begin{equation}\label{EQ:div:002p5}
(\chi_1, \partial_x p)_T = \chi_1(M_1) \int_T \partial_xp dT + R_{11}(T),
\end{equation}
where $M_i$ is the center of the edge $e_i$ and $R_{11}(T) = \int_T E_1(y)\partial_y \chi_1 \partial_{xy}^2 p dT$ with the kernel satisfying $|E_1(y)|\leq C h^2$. It is easy to see that $\partial_y \chi_1 = \partial_y \S(v_1)$. Thus, the error term $R_{11}(T)$ has the following estimate:
\begin{equation}\label{EQ:div:002}
|R_{11}(T)| \leq C h^2 \|\nabla \S(v_1)\|_T \|\nabla^2 p\|_T.
\end{equation}
For the first term on the right-hand of \eqref{EQ:div:002p5}, we apply the trapezoidal rule in the $x$-direction to obtain
\begin{equation}\label{EQ:div:003}
\chi_1(M_1) \int_T \partial_xp dT = \frac{h}{2} \chi_1(M_1) \int_{e_1} \partial_x p dy + \frac{h}{2} \chi_1(M_2) \int_{e_2} \partial_x p dy + R_{12}(T),
\end{equation}
where we have used $\chi_1(M_1) = \chi_1(M_2)$ and
$$
R_{12}(T) = \chi_1(M_1) \int_T E_2(x) \partial_{xx}^2 p dT
$$
with the kernel function satisfying $|E_2(x)|\leq C h$. The remainder term $R_{12}(T)$ can then be bounded as follows:
\begin{equation}\label{EQ:div:004}
|R_{12}(T)| \leq C h^2 |\chi_1(M_1)|\ \|\nabla^2 p\|_T.
\end{equation}
Substituting \eqref{EQ:div:003} into \eqref{EQ:div:002p5} yields the following:
\begin{equation}\label{EQ:div:005}
(\chi_1, \partial_x p)_T = \frac{h}{2} \chi_1(M_1) \int_{e_1} \partial_x p dy + \frac{h}{2} \chi_1(M_2) \int_{e_2} \partial_x p dy + R_{11}(T)+R_{12}(T).
\end{equation}
By introducing the following function
$$
\rho_b^{(1)}=\left\{
\begin{array}{rl}
\frac{1}{|e_1|} \int_{e_1} \partial_x p dy\qquad &\mbox{on } e_1,\\
\frac{1}{|e_2|} \int_{e_2} \partial_x p dy\qquad &\mbox{on } e_2,\\
0 \qquad &\mbox{on $e_3$ and $e_4$},
\end{array}
\right.
$$
we may rewrite \eqref{EQ:div:005} as follows:
\begin{equation}\label{EQ:div:006}
\begin{split}
(\chi_1, \partial_x p)_T = &\ \frac{h}{2} \langle \chi_1, \rho_b^{(1)}\rangle_\pT + R_{11}(T)+R_{12}(T)\\
=& \ \frac{h}{2} \langle \S(v_1) - v_1, \rho_b^{(1)}\rangle_\pT + R_{11}(T)+R_{12}(T).
\end{split}
\end{equation}

Analogously, for the second term on the right-hand side of \eqref{EQ:div:001}, there are remainder terms $R_{21}(T)$ and $R_{22}(T)$ such that
\begin{equation}\label{EQ:div:007}
\begin{split}
(\chi_2, \partial_y p)_T =\ \frac{h}{2} \langle \S(v_2) - v_2, \rho_b^{(2)}\rangle_\pT + R_{21}(T)+R_{22}(T),
\end{split}
\end{equation}
where
$$
\rho_b^{(2)}=\left\{
\begin{array}{rl}
\frac{1}{|e_3|} \int_{e_3} \partial_y p dx\qquad &\mbox{on } e_3,\\
\frac{1}{|e_4|} \int_{e_4} \partial_y p dx\qquad &\mbox{on } e_4,\\
0 \qquad &\mbox{on $e_1$ and $e_2$},
\end{array}
\right.
$$

By setting $\rho_b=(\rho_b^{(1)}, \rho_b^{(2)})^t$, we have from combining \eqref{EQ:div:001} with \eqref{EQ:div:006} and \eqref{EQ:div:007} that
\begin{equation}\label{EQ:div:008}
\begin{split}
 \langle(\S(\bm{v})-\bm{v})\cdot \bm{n},p- Q_h p\rangle_{\partial T}
=  \frac{h}{2} \langle \S(\bm{v}) - \bm{v}, \rho_b\rangle_\pT + R(T),
\end{split}
\end{equation}
where $R(T) = \sum_{i,j=1}^2 R_{ij}(T)$ is the remainder satisfying
\begin{equation}\label{EQ:div:009}
\begin{split}
|R(T)| \leq & \ C h^2 \left(\|\nabla \S(\bm{v})\|_T + \sum_{i=1}^4 |(\S(\bm{v})-\bm{v})(M_i)|\right) \|\nabla^2 p\|_T\\
\leq & \ C h^2 \left(\|\nabla_w \bm{v}\|_T + S_T(\bm{v}, \bm{v})^{1/2}\right) \|\nabla^2 p\|_T.
\end{split}
\end{equation}
Here we have used the estimate \eqref{EQ:S-Property:0212} and the expression \eqref{Def-ST-poly} in the second inequality.
Summing over all the elements in $\T_h$ gives
\begin{equation}\label{EQ:div:010}
\sum_{T\in\T_h} \langle(\S(\bm{v})-\bm{v})\cdot \bm{n},p- Q_h p\rangle_{\partial T}
=  \frac{h}{2} \sum_{T\in\T_h} \langle \S(\bm{v}) - \bm{v}, \rho_b\rangle_\pT + R,
\end{equation}
where
\begin{equation}\label{EQ:div:011}
\begin{split}
|R| = \left| \sum_{T\in\T_h} R(T) \right|
\leq C h^2 \left(\|\nabla_w\bm{v}\| + S(\bm{v}, \bm{v})^{1/2}\right) \|\nabla^2 p\|.
\end{split}
\end{equation}

As to the first two terms on the right-hand side of \eqref{equ.error.3}, it has been shown in \cite{LiDanWW} that exists another function $\sigma_b=(\sigma_b^{(1)},  \sigma_b^{(2)})^t$ and a remainder $R_5$ such that
\begin{equation}\label{EQ:div:012}
\begin{split}
&\ \kappa {S}(Q_b\bm{u},\bm{v})
+ \sum_{T}\langle \frac{\partial \bm{u}}{\partial \bm{n}}- \mathbb{Q}_h(\nabla \bm{u})\cdot\bm{n},\S(\bm{v})-\bm{v}\rangle_{\partial T}\\
& = \frac{h}{2} \sum_{T\in\T_h} \langle \bm{v} - \S(\bm{v}), \sigma_b\rangle_\pT + R_5,
\end{split}
\end{equation}
where
\begin{equation}\label{EQ:div:013}
|R_5|
\leq C h^2 \left(\|\nabla_w\bm{v}\| + S(\bm{v}, \bm{v})^{1/2}\right) \|\nabla^3 \bm{u}\|.
\end{equation}

Combining \eqref{equ.error.3} with \eqref{EQ:div:012} and \eqref{EQ:div:010} gives
\begin{equation}\label{EQ:div:015}
\begin{split}
\varphi_{\bm{u},p}(\bm{v}) = &\ \frac{h}{2} \sum_{T\in\T_h} \langle \bm{v} - \S(\bm{v}), \sigma_b+\rho_b\rangle_\pT + R_5 -R\\
= & \ \frac{h^2}{2} S(\sigma_b+\rho_b, \bm{v}) + R_5 -R\\
= & \ \frac{h^2}{2\kappa} \left[ a(\sigma_b+\rho_b, \bm{v}) - (\nabla_w( \sigma_b+\rho_b), \nabla_w \bm{v})\right] + R_5 - R,
\end{split}
\end{equation}
where we have used \eqref{EQ:June30:001} in the last equation. By letting $\tilde {\bm{e}} = \bm{e} - \frac{h^2}{2\kappa} (\sigma_b+\rho_b)$, the error equations (\ref{equ.error.1})-(\ref{equ.error.2}) can be rewritten as
\begin{eqnarray}
a(\tilde{\bm{e}},\bm{v})-b(\bm{v},\eta) &=& R_5-R - \frac{h^2}{2\kappa} (\nabla_w( \sigma_b+\rho_b), \nabla_w \bm{v}), \label{equ.error.1:001}\\
b(\tilde{\bm{e}},w)&=&- \frac{h^2}{2\kappa} b(\sigma_b+\rho_b,w),\label{equ.error.2:001}
\end{eqnarray}
where the right-hand sides have the following bound:
\begin{equation}\label{EQ:div:018}
\begin{split}
&\left| R_5-R - \frac{h^2}{2\kappa} (\nabla_w( \sigma_b+\rho_b), \nabla_w \bm{v})\right| \leq C h^2 (\|\bm{u}\|_3 + \|p\|_2) a(\bm{v},\bm{v})^{1/2},\\
&\left| \frac{h^2}{2\kappa} b(\sigma_b+\rho_b,w)\right| \leq C h^2 (\|\bm{u}\|_3 + \|p\|_2) \|w\|.
\end{split}
\end{equation}
The equations \eqref{equ.error.1:001}-\eqref{equ.error.2:001} with $\O(h^2)$ load functions on the right-hand sides lead immediately to the superconvergence estimate \eqref{equ.thm.error.2} by selecting a particular test function $\bm{v}$. The estimate for $\eta$ is given by the inf-sup condition. The drop of half order is due to the fact that the small perturbation term $\frac{h^2}{2\kappa} (\sigma_b+\rho_b)$ applied to the error function $\bm{e}$ may not vanish on the boundary. Details on how the technique works can be found in \cite{LiDanWW}.
\end{proof}

\section{Numerical Experiments}\label{numerical-experiments}
In this section, we numerically verify the theoretical error estimates developed in the previous sections for the finite difference scheme \eqref{equ.Stokes_FDS} when applied to the Stokes problem \eqref{Equ.Stokes}.
The following metrics are used to measure the magnitude of the error:
\begin{eqnarray*}
&&\text{$L^2$-norm for the velocity $w=u,v$: }\\
&&\|w_b-w\|_{0}^2=\sum_{i=1}^{n+1}\sum_{j=1}^{n}h^2|w_{i-\frac{1}{2},j}-w(x_{i-\frac{1}{2}},y_j)|^2  +\sum_{i=1}^{n}\sum_{j=1}^{n+1}h^2|w_{i,j-\frac{1}{2}}-w(x_{i},y_{j-\frac{1}{2}})|^2,\\
&&\text{$H^1$-norm for the velocity $w=u,v$: }\\
&&\|w_b-w\|_1^2=\sum_{i=1}^{n}\sum_{j=1}^{n} h^2\left|\frac{w_{i+\frac{1}{2},j}-w_{i-\frac{1}{2},j}}{h} - \frac{\partial w}{\partial x} (x_{i},y_{j})\right|^2 \nonumber\\
&&\qquad\qquad\qquad  + \sum_{i=1}^{n}\sum_{j=1}^{n} h^2\left|\frac{w_{i,j+\frac{1}{2}}-w_{i,j-\frac{1}{2}}}{h} - \frac{\partial w}{\partial y} (x_{i},y_{j})\right|^2,\\
&&\text{$L^2$-norm for the pressure $p$: }\\
&&\|p_h-p\|_0^2= \sum_{i=1}^{n}\sum_{j=1}^{n} h^2|p_{i,j}-p (x_{i},y_{j})|^2,
\end{eqnarray*}
For simplicity, only uniform square meshes are employed in our numerical experiments.

\subsection{Test Case 1}
The computational domain in the first test case is given by $\Omega=(0,\pi)\times (0,\pi)$, which is partitioned into $n\times n$ small squares of equal-size. The exact solution of the model problem is given by
\begin{equation*}
\left \{
\begin{array}{lll}
\bm{u} =\left(
\begin{array}{lll}
&\sin^2(x)\cos(y)\sin(y) &\\
-&\cos(x)\sin(x)\sin^2(y)&
\end{array}
\right),\\
~\\
p=\cos(x)\cos(y).
\end{array}
\right.
\end{equation*}
The right-hand side function and the Dirichlet boundary data are chosen to match the exact solution. Note that this example has been considered in \cite{RXZZ-JCAM-2016}.

The numerical approximation pressure field is shown in Fig. \ref{Fig.sub1.2.lp}, while the velocity vector and magnitude filds are plotted in Fig. \ref{Fig.sub1.2.2v}.
\begin{figure}[H]
\centering
\subfigure[Pressure field]{
\label{Fig.sub1.2.lp}
\includegraphics [width=0.45\textwidth]{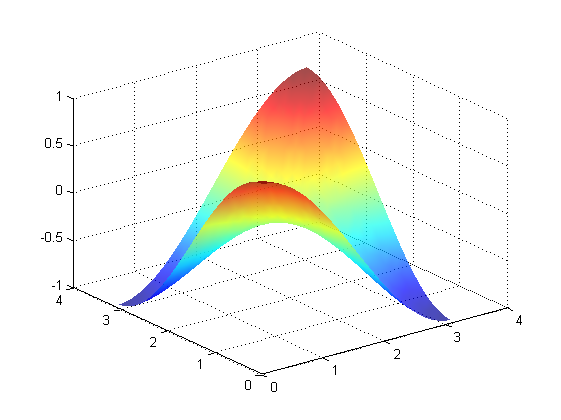}
}
\subfigure[Velocity field]{
\label{Fig.sub1.2.2v}
\includegraphics [width=0.45\textwidth]{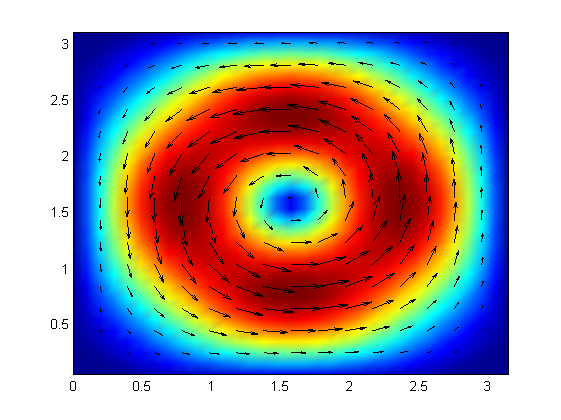}
}

\caption{Numerical results of test case 1 obtained with the SWG scheme using uniform rectangular meshes with $h=1/32$.}
\label{fig:testcase1}
\end{figure}

Table \ref{table4} illustrates the numerical performance of the finite difference scheme \eqref{equ.Stokes_FDS}. The numerical velocity is clearly converging at the optimal order of $r=2$ in the $L^2$ norm. For the numerical pressure, the error at the cell center is decreasing at the order of $r=2$, and the same conclusion can be obtained for the numerical velocity in a discrete $H^1$ norm defined by using the cell centers. This implies that the pressure and the numerical velocity are of  superconvergent to the exact solution at the center of cells. The numerical results outperforms the theoretical prediction of $r=1.5$ shown in Theorem \ref{thm.error.2}.

{\small
\begin{table}[h]
\begin{center}
\caption{Error and convergence performance of the finite difference scheme \eqref{equ.Stokes_FDS} for the Stokes equation on uniform meshes of test case 1. $r$ refers to the order of convergence in $O(h^r)$.}\label{table4}
{\tiny\begin{tabular}{||c|cc|cc|cc|cc|cc||}
\hline
$n$ & $\|u_h-u \|_{0}$ & $r=$ & $\|u_h-u\|_{1}$ & $r=$ &$\|v_h-v \|_{0}$& $r=$ &$\|v_h-v\|_{1}$ & $r=$ &$\|p_h-p \|_{0}$& $r=$ \\
\hline
8     & 2.35e-02  &  0.00   &   5.90e-02  &  0.00  &    5.69e-02 &   0.00  &    6.61e-02  &  0.00   &   1.48e-01 &   0.00\\
16    & 6.26e-03  &  1.91   &   1.64e-02  &  1.85  &    1.53e-02 &   1.90  &    1.92e-02  &  1.78   &   4.29e-02 &   1.79\\
32    & 1.60e-03  &  1.97   &   4.25e-03  &  1.95  &    3.89e-03 &   1.97  &    5.01e-03  &  1.94   &   1.13e-02 &   1.92\\
64    & 4.01e-04  &  1.99   &   1.08e-03  &  1.98  &    9.78e-04 &   1.99  &    1.27e-03  &  1.98   &   2.88e-03 &   1.97\\
\hline
\end{tabular}}
\end{center}
\end{table}
}

\subsection{Test Case 2} In this test case, the analytical solution is given by
\begin{equation*}
\left \{
\begin{array}{lll}
\bm{u} =\left(
\begin{array}{lll}
-&256x^2(x-1)^2y(y-1)(2y-1)&\\
 &256y^2(y-1)^2x(x-1)(2x-1)&
\end{array}
\right),\\
~\\
p= 150(x-0.5)(y-0.5),
\end{array}\right.
\end{equation*}
and the domain is the unit square $\Omega=(0,1)\times(0,1)$.
The right-hand side function and the Dirichlet boundary data are chosen to match the exact solution. Note that this example has been considered in \cite{LS-JSC-2015}.

The numerical approximation pressure field is shown in Fig. \ref{Fig.sub2.2.lp}, while the velocity vector and magnitude filds are plotted in Fig. \ref{Fig.sub2.2.2v}.
\begin{figure}[H]
\centering
\subfigure[Pressure field]{
\label{Fig.sub2.2.lp}
\includegraphics [width=0.45\textwidth]{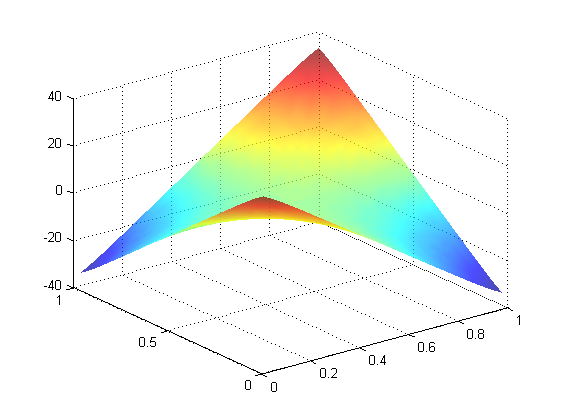}
}
\subfigure[Velocity field]{
\label{Fig.sub2.2.2v}
\includegraphics [width=0.45\textwidth]{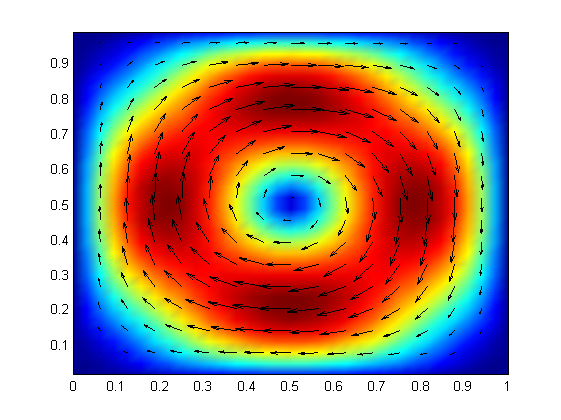}
}

\caption{Numerical results of test case 2 obtained with the SWG scheme using uniform rectangular meshes with $h=1/32$.}
\label{fig:testcase2}
\end{figure}

Table \ref{table5} illustrates the numerical performance of the finite difference scheme \eqref{equ.Stokes_FDS} for Test Case 2. It can be seen that the numerical velocity is converging at the optimal order of $r=2$ in the usual $L^2$ norm. For the pressure approximation, the error at the cell centers seems to be decreasing  at the rate of $O(h^{1.8})$. It appears that the numerical velocity is also converging at the order of $r=1.8$ in the discrete $H^1$ norm defined on the cell centers. The results clearly indicate a superconvergence of the finite difference scheme for the velocity in $H^1$ and the pressure in $L^2$ on cell centers. Once again, the numerical results outperforms the theoretical prediction of $r=1.5$ shown in Theorem \ref{thm.error.2}.

\begin{table}[h]
\begin{center}
\caption{Error and convergence performance of the finite difference scheme \eqref{equ.Stokes_FDS} for the Stokes equation on uniform meshes of test case 2. $r$ refers to the order of convergence in $O(h^r).$ }\label{table5}
{\tiny
\begin{tabular}{||c|cc|cc|cc|cc|cc||}
\hline
$n$ & $\|u_h-u \|_{0}$ & $r=$ & $\|u_h-u\|_{1}$ & $r=$ &$\|v_h-v \|_{0}$& $r=$ &$\|v_h-v\|_{1}$ & $r=$ &$\|p_h-p \|_{0}$& $r=$ \\
\hline
8   &  1.03e-01  & 0.00   &  6.26e-01  & 0.00  &   7.17e-02  & 0.00  &   4.78e-01 &  0.00  &   1.39e+00  & 0.00\\
16  &  2.90e-02  & 1.82   &  1.97e-01  & 1.67  &   2.07e-02  & 1.79  &   1.60e-01 &  1.57  &   4.68e-01  & 1.58\\
32  &  7.55e-03  & 1.94   &  5.73e-02  & 1.78  &   5.43e-03  & 1.93  &   4.88e-02 &  1.72  &   1.43e-01  & 1.71\\
64  &  1.91e-03  & 1.98   &  1.60e-02  & 1.84  &   1.38e-03  & 1.98  &   1.41e-02 &  1.79  &   4.14e-02  & 1.79\\
\hline
\end{tabular} }
\end{center}
\end{table}

\subsubsection{Test Case 3: Lid driven Cavity}
The lid-driven cavity problem is a benchmark test case for Stokes flow, which has been tested in many existing literatures, including \cite{LTF-IJNMF-1995,LIU-SIAM-NA-2011,WWY-SIAM-JSC-2009,RXZZ-JCAM-2016}.
In this test case, a uniform mesh with meshsize $h=1/32$ is employed in our new finite difference scheme.
The source term is $\bm{f} =\bm{0}$, and the Dirichlet boundary condition is given as $\bm{u} = (1,0)^{T}$ for $y=1, x\in(0,1)$ and $\bm{u} = (0,0)^{T}$ on the rest of the boundary.
The exact solution of lid-driven problem is not known, but the solution is known to have singularity at the corner points $(0,1)$ and $(1,1)$.
The velocity field and streamlines are plotted in Fig. \ref{Fig.sub.2.lv} and Fig.\ref{Fig.sub.2.2s}, respectively. The pressure contour plot is shown in Fig. \ref{fig:mesh:2}. The shape of streamlines is similar to the result given in \cite{LTF-IJNMF-1995,LIU-SIAM-NA-2011,WWY-SIAM-JSC-2009,RXZZ-JCAM-2016}, and the results look quite reasonable.
\begin{figure}[H]
\centering
\subfigure[Velocity field]{
\label{Fig.sub.2.lv}
\includegraphics [width=0.40\textwidth]{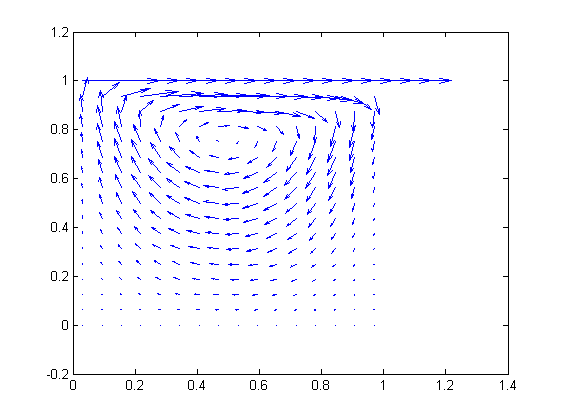}
}
\subfigure[Streamlines]{
\label{Fig.sub.2.2s}
\includegraphics [width=0.40\textwidth]{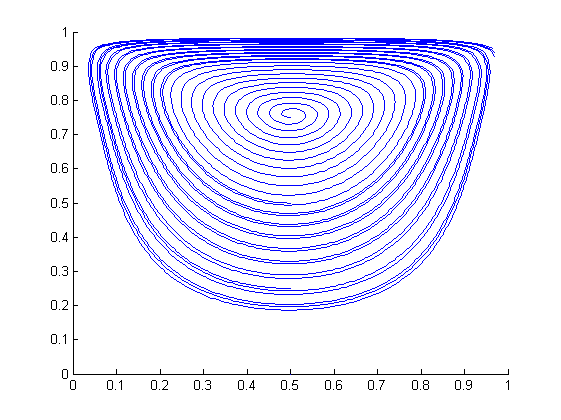}
}
%\subfigure[Pressure contours]{
%\label{Fig.sub.2.3p}
%\includegraphics [width=0.45\textwidth]{lid-driven-cavity-pressure.png}
%}
\caption{Numerical results of lid-driven problem obtained with the SWG scheme using uniform rectangular meshes with $h=1/32$.}
\label{fig:mesh}
\end{figure}

\begin{figure}[H]
\centering
%\subfigure[Velocity field]{
%\label{Fig.sub.2.lv}
%\includegraphics [width=0.4\textwidth]{lid-driven-cavity-vec.png}
%}
%\subfigure[Streamlines]{
%\label{Fig.sub.2.2s}
%\includegraphics [width=0.45\textwidth]{lid-driven-cavity-stream.png}
%}
%\subfigure[Pressure contours]{
%\label{Fig.sub.2.3p}
\includegraphics [width=0.40\textwidth]{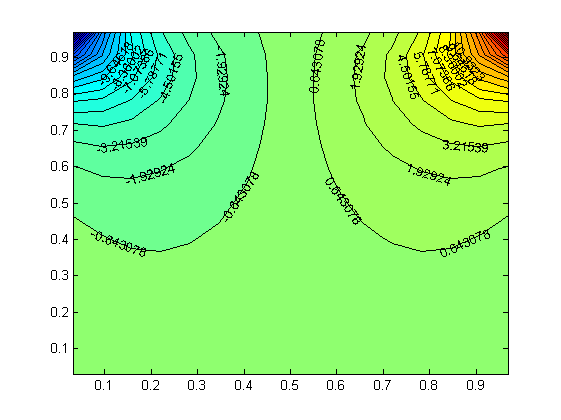}
%}
\caption{Numerical results -- Pressure contours of lid-driven problem obtained with the SWG scheme using uniform rectangular meshes with $h=1/32$.}
\label{fig:mesh:2}
\end{figure}

%\newpage
%\vfill\eject


\begin{thebibliography}{99}

\bibitem{Babuska1973} I. Babu$\breve{s}$ka.
\newblock{ The finie element method with Lagrange multipliers},
\newblock {\em Numer. Math.}, 20 : 173-192, 1973.

\bibitem{Brenner2007}
S. Brenner and R. Scott.
\newblock {Mathematical theory of finite element methods},
\newblock{\em Springer}, 2007.

\bibitem{Brezzi1974}F. Brezzi.
\newblock {On the existence, uniqueness, and approximation of saddle point problems arising from Lagrange multipliers,}
\newblock{\em RAIRO}, 8: 129-151, 1974.

\bibitem{Brezzi1991}
F. Brezzi and M. Fortin.
\newblock {Mixed and hybrid finite element methods},
\newblock{\em Springer-Verlag}, New York, 1991.

\bibitem{Girault1986}
V. Girault and P. A. Raviart.
\newblock {Finite element methods for Navier-Stokes equations: theory and algorithms},
\newblock{\em Springer-Verlag}, Berlin, 1986.

\bibitem{Gunzburger1989}
M D. Gunzburger.
\newblock {Finite element methods for viscous incompressible flows: a guide to theory, practice, and algorithms},
\newblock{\em Academic}, San Diego, 1989.

\bibitem{Han_Wu_1998}
H. Han, X. Wu.
\newblock {A new mixed finite element formulation and the MAC method for the Stokes equations[J]},
\newblock{\em SIAM Journal on Numerical Analysis}, 35(2): 560-571, 1998.

\bibitem{Kanschat2008}
G. Kanschat.
\newblock {Divergence-free discontinuous Galerkin schemes for the Stokes equations and the MAC scheme[J]},
\newblock{\em International journal for numerical methods in fluids}, 56(7): 941-950, 2008.

\bibitem{LS-JSC-2015}
J. Li and S. Sun.
\newblock{The superconvergence phenomenon and proof of the MAC scheme for the Stokes equations on non-uniform rectangular meshes},
\newblock{\em Journal of Scientific Computing}, 65(1): 341-362, 2015.

\bibitem{LTF-IJNMF-1995}
M. Li, T. Tang, B. Fornberg.
\newblock{ A compact fourth-order finite difference scheme for the steady incompressible Navier-Stokes equations}.
\newblock{\em International Journal for Numerical Methods in Fluids}, 20(10): 1137-1151, 1995.

\bibitem{LiDanWW}
D. Li, C. Wang, J. Wang.
\newblock{Superconvergence of the gradient approximation for weak Galerkin finite element methods on nonuniform rectangular partitions}, https://arxiv.org/pdf/1804.03998v2.pdf.

\bibitem{LIU-SIAM-NA-2011}
J. Liu.
\newblock{Penalty-factor-free discontinuous Galerkin methods for 2-dim Stokes problems},
\newblock{\em SIAM Journal on Numerical Analysis}, 49(5): 2165-2181, 2011.

\bibitem{Minev_2008}
P. D. Minev.
\newblock{Remarks on the links between low-order DG methods and some finite-difference schemes for the Stokes problem[J]}.
\newblock{\em International Journal for Numerical Methods in Fluids}, 58(3): 307-317, 2008.

\bibitem{Nicolaides_1989}
R. Nicolaides.
\newblock{Flow discretization by complementary volume techniques[C]},
\newblock{\em 9th Computational Fluid Dynamics Conference,} 1989.

\bibitem{Strikwerda_1984}
J. C. Strikwerda.
\newblock{Finite difference methods for the Stokes and Navier-Stokes equations[J]},
\newblock{\em SIAM Journal on Scientific and Statistical Computing}, 5(1): 56-68, 1984.

\bibitem{WangWang_2016}
C. Wang and J. Wang.
\newblock{A primal-dual weak Galerkin finite element method for second elliptic equations in non-divergence form},
{\it Math. Comp.}, DOI: https://doi.org/10.1090/mcom/3220. June 2017.

\bibitem{WWY-SIAM-JSC-2009}
J. Wang, Y. Wang, X. Ye.
\newblock{A robust numerical method for Stokes equations based on divergence-free H (div) finite element methods},
\newblock{\em SIAM Journal on Scientific Computing}, 31(4): 2784-2802, 2009.

\bibitem{WangYe_2013}
J. Wang and X. Ye.
\newblock{A weak Galerkin mixed finite element method for second-order ellliptic problems.}
available at arXiv: 1104.2897vl.
{\it J. Comp. and Appl. Math.}, {241}, 103-115, 2013.

\bibitem{wy3655}
{\sc J. Wang and X. Ye}. {\em A weak Galerkin mixed finite element
method for second-order elliptic problems}, Math. Comp., 83 (2014), pp. 2101-2126.

\bibitem{wy-stokes} {\sc J. Wang and X. Ye}, {\em A weak
Galerkin finite element method for the Stokes equations}.
arXiv:1302.2707v1. Advances in Computational Mathematics,
 Volume 42, Issue 1, pp. 155-174, 2016. DOI 10.1007/s10444-015-9415-2.

\bibitem{RXZZ-JCAM-2016}
R. Wang, X. Wang, Q. Zhai, R. Zhang.
\newblock{A weak Galerkin finite element scheme for solving the stationary Stokes equations},
\newblock{\em Journal of Computational and Applied Mathematics}, 302: 171-185, 2016.

\bibitem{Welch-MAC}
J. E. Welch, F. H. Harlow, J. P. Shannon, et al.
\newblock{The MAC method. A computing technique for solving viscous, incompressible, transient fluid-flow problems involving free surfaces}, 1966.


\end{thebibliography}
\end{document}